\documentclass[review, compress, 3p]{elsarticle}
\usepackage{lineno,hyperref}
\usepackage{amssymb,amsfonts}
\usepackage{amsmath,amsthm}
\usepackage{algorithm,algorithmic}
\usepackage{caption}
\usepackage{graphicx,subfigure}
\usepackage{float,rotating}
\usepackage{booktabs,array}
\usepackage{multirow,multicol}
\usepackage{lscape}
\usepackage{epstopdf}
\usepackage{color}
\usepackage{setspace}
\usepackage{natbib}
\usepackage{bm}
\usepackage{url,doi}

\makeatletter
\@addtoreset{equation}{section}
\makeatother

\newtheorem{theorem}{Theorem}[section]
\newtheorem{lemma}{Lemma}[section]

\newtheorem{proposition}{Proposition}[section]

\journal{arXiv}

\begin{document}

\begin{frontmatter}

\title{An adaptive low-rank splitting approach for the extended Fisher--Kolmogorov equation}

\author[address1]{Yong-Liang Zhao}
\ead{ylzhaofde@sina.com}

\author[address2]{Xian-Ming Gu\corref{correspondingauthor}}
\cortext[correspondingauthor]{Corresponding author}
\ead{guxianming@live.cn, guxm@swufe.edu.cn}



\address[address1] {School of Mathematical Sciences,
Sichuan Normal University,
Chengdu, Sichuan 610068, P.R. China}
\address[address2] {School of Mathematics, 
	Southwestern University of Finance and Economics, Chengdu, Sichuan 611130, P.R. China}

\begin{abstract}
The extended Fisher--Kolmogorov (EFK) equation has been used
to describe some phenomena in physical, material and biology systems.
In this paper, we propose a full-rank splitting scheme and a rank-adaptive splitting approach
for this equation.
We first use a finite difference method to approximate the space derivatives.
Then, the resulting semi-discrete system is split into two stiff linear parts
and a nonstiff nonlinear part.
This leads to our full-rank splitting scheme.
The convergence and the maximum principle of the proposed scheme are proved rigorously.
Based on the frame of the full-rank splitting scheme,
a rank-adaptive splitting approach for obtaining a low-rank solution of the EFK equation.
Numerical examples show that our methods are robust and accurate.
They can also preserve energy dissipation and the discrete maximum principle.
\end{abstract}


\end{frontmatter}


\section{Introduction}
\label{sec1}

The extended Fisher--Kolmogorov (EFK) equation
\begin{equation*}
\partial_t u = - \kappa \Delta^2 u + \Delta u - F'(u)
\end{equation*}
proposed by Coullet et al.~\cite{coullet1987nature}, Dee and Van Saarloos \cite{van1987dynamical,van1988front,dee1988bistable}
is important in describing diverse physical, chemical and biological phenomena.
Such as the diffusion of domain walls in liquids crystals \cite{guozhen1982experiments},
traveling waves in reaction-diffusion systems \cite{aronson1978multidimensional,ahlers1983vortex},
the mesoscopic model of a phase transition
in a binary system near Lipschitz points \cite{hornreich1975critical} and
the spatio-temporal turbulence in a bistable system \cite{coullet1987nature}.
Here $\Delta$ is the Laplacian, $F(u) = \frac{1}{4} \left( u^2 - 1 \right)^2$,
the stabilizing parameter $\kappa$ is a positive constant.
Set $\kappa = 0$ in the above equation, the Fisher--Kolmogorov equation \cite{Kolmogorov37} is obtained.
Many applications of the Fisher--Kolmogorov equation can be found in
brain tumor dynamics \cite{belmonte2014effective} and
reaction-diffusion in chemistry \cite{adomian1995fisher}.

In this paper, we concentrate on the following EFK model equipped with the periodic boundary condition \cite{sun2022convex}:
\begin{equation}
\begin{cases}
\partial_t u(x,y,t) = - \kappa \Delta^2 u(x,y,t) + \Delta u(x,y,t) + f(u(x,y,t)),
& \textrm{in} \quad \Omega \times (0,T], \\
u(x,y,0) = u_0(x,y), &  \textrm{in} \quad \Omega,
\end{cases}
\label{eq1.1}
\end{equation}
where $\Omega = [x_L,x_R] \times [y_L,y_R] \subset \mathbb{R}^2$ and $f(u) = -F'(u) = u - u^3$.
The energy of Eq.~\eqref{eq1.1} is defined as
\begin{equation*}
	E(u) = \int_{\Omega} \left( \frac{\kappa}{2} \left| \Delta u \right|^2 +
	\frac{1}{2} \left| \nabla u \right|^2 + F(u) \right) d \textbf{x}.
\end{equation*}

Peletier and Troy \cite{peletier1997spatial} studied stationary periodic solutions of Eq.~\eqref{eq1.1}.
Obviously, the EFK equation is a fourth-order nonlinear partial differential equation
in which the analytical solution is often unavailable.
Thus, some numerical methods are developed for solving Eq. \eqref{eq1.1},
see \cite{sun2022convex,danumjaya2005orthogonal,danumjaya2006numerical,kadri2011second,
	khiari2011finite,gudi2013fully,he2016norm,dehghan2021local,abbaszadeh2020error,ismail2022three,li2022new}.
Danumjaya and Pani \cite{danumjaya2005orthogonal} proposed an orthogonal cubic spline collocation method
for the one-dimensional (1D) EFK equation.
Then, for the same model, they proposed three numerical schemes by
using the $C^1$-conforming finite element method in space \cite{danumjaya2006numerical}.
Kadri and Omrani \cite{kadri2011second} presented a nonlinear Crank-Nicolson finite difference scheme
for the 1D EFK equation. The uniqueness and convergence of this scheme are also studied.
Later, Khiari and Omrani \cite{khiari2011finite} extended the scheme for the two-dimensional (2D) EFK equation.
He \cite{he2016norm} designed a three-level linearly implicit finite difference scheme
for the EFK equations in both 1D and 2D.
Then, in order to obtain higher-order numerical results for the EFK equations,
Ismail et al.~\cite{ismail2022three} proposed a fourth-order in space
three level linearized difference scheme.
Sun et al.~\cite{sun2022convex} presented a convex splitting variable-step BDF2 scheme for Eq.~\eqref{eq1.1}.
The uniqueness and energy stable of the scheme are analyzed.

For unbearably large systems, many researchers prefer to use low-rank approximation techniques
since they can reduce the computational complexity. However, the above mentioned literatures are not consider low-rank approaches for the (high-dimensional) EFK equation.
In this paper, based on the idea of the dynamical low-rank approximation (DLRA) \cite{koch2007dynamical},
we first attempt to find low-rank solutions of \eqref{eq1.1}.
Inspired by the Dirac--Frenkel time-dependent variational principle,
Koch and Lubich \cite{koch2007dynamical} first proposed the DLRA for time-dependent data matrices
and matrix differential equations.
Recently, it has been widely used,
see \cite{nonnenmacher2008dynamical,ostermann2019convergence,Einkemmer2019Vlasov,
	zhao2021low,grasedyck2013literature,einkemmer2022robust}
and references therein.
Lubich and Oseledets \cite{lubich2014projector} proposed
a projector-splitting integrator for solving the DLRA resulting matrix differential equations numerically.
Kieri and Vandereycken \cite{kieri2019projection} considered the DLRA
on the manifold of fixed-rank matrices and tensor trains.
Ceruti and Lubich \cite{Ceruti2022unconventional} proposed
a new low-rank integrator for the DLRA.
Compared with the projector-splitting integrator \cite{lubich2014projector},
their integrator avoids the backward time integration substep and offers more parallelism.
The mentioned methods are fixed-rank integrators.
It is not suitable for the case that the optimal rank required
for a given approximation accuracy is unknown beforehand and may vary greatly over time.
To handle this problem, Ceruti et al.~\cite{Ceruti2022adaptiveDLR} designed
a rank-adaptive integrator for the DLRA.
For solving evolution problems by using the DLRA in a parallel-in-time pattern,
Carrel et al.~\cite{Carrel2022DLRparareal} presented a low-rank Parareal algorithm.

The main contributions of this work can be summarized as follows:

(i) We split Eq.~\eqref{eq1.1} into three subproblems, then a full-rank splitting (FRS) scheme is established.
The convergence of this scheme is analyzed.

(ii) Based on this full-rank splitting scheme,
we propose a rank-adaptive low-rank approach for Eq.~\eqref{eq1.1}.
To the best of our knowledge, no literature considers such low-rank approaches for \eqref{eq1.1}.

The rest of this paper is organized as follows.
Section \ref{sec2} derives our full-rank splitting scheme
and proves the convergence and the maximum principle preserving of the scheme.
In Section \ref{sec3}, for finding a low-rank approximation of Eq.~\eqref{eq1.1},
a rank-adaptive low-rank splitting approach is proposed.
Several numerical experiments are reported in Section \ref{sec4}.
Finally, some conclusions are drawn in Section \ref{sec5}.
\section{A full-rank splitting scheme and its convergence}
\label{sec2}

In this section, we derive the FRS scheme for Eq.~\eqref{eq1.1} and prove its convergence.
\subsection{A full-rank splitting scheme}
\label{sec2.1}

For three given positive integers $M$, $N_x$ and $N_y$, let $\tau = \frac{T}{M}$,
$h_x = \frac{x_R - x_L}{N_x}$ and $h_y = \frac{y_R - y_L}{N_y}$
be time step size, spatial grid sizes in $x$ and $y$ directions, respectively.
Then, the domain $\Omega \times [0,T]$ can be
covered by the mesh $\omega_{h \tau} = \bar{\omega}_{h} \times \bar{\omega}_{\tau}$,
where $\bar{\omega}_{\tau} = \left\{ t_k = k \tau, ~k = 0, 1, \ldots, M \right\}$ and
\begin{equation*}
\bar{\omega}_{h} = \left\{ x_h = (x_i,y_j) \mid x_i = x_L + i h_x,
y_j = y_L + j h_y, ~i = 0, 1, \ldots, N_x; j = 0, 1, \ldots, N_y \right\}.
\end{equation*}

The second-order and fourth-order derivative terms in \eqref{eq1.1} at $(x,y) = (x_i,y_j)$
can be approximated as
\begin{equation*}
\begin{split}
\Delta u(x_i, y_j, t) & =
\frac{u(x_{i - 1}, y_j, t) - 2 u(x_i, y_j, t) + u(x_{i + 1}, y_j, t)}{h_x^2} + \\
&\qquad \frac{u(x_i, y_{j - 1}, t) - 2 u(x_i, y_j, t) + u(x_i, y_{j + 1}, t)}{h_y^2}
+ \mathcal{O}(h_x^2 + h_y^2) \\
& \triangleq \delta_s u(x_i, y_j, t) + \mathcal{O}(h_x^2 + h_y^2),
\end{split}
\end{equation*}
\begin{equation*}
		\Delta^2 u(x_i, y_j, t)  = \delta_s \left( \delta_s u(x_i, y_j, t) \right)
		+ \mathcal{O}(h_x^2 + h_y^2)
		 \triangleq \delta_s^2 u(x_i, y_j, t) + \mathcal{O}(h_x^2 + h_y^2),
\end{equation*}
respectively.

Let $u_{ij}(t)$ be the numerical approximation of $u(x_i, y_j, t)$.
Then, we get the semi-discrete scheme of \eqref{eq1.1}:
\begin{equation*}
\frac {d u_{ij}(t)}{d t} = - \kappa \delta_s^2 u_{ij}(t) + \delta_s u_{ij}(t) + f(u_{ij}(t)).
\end{equation*}

Two identity matrices with orders $N_x$ and $N_y$ are denoted as $I_x$ and $I_y$, respectively.
Let
\begin{align*}
& \bm{u}(t) = \left[u_{11}(t), \ldots, u_{N_x,1}(t), u_{12}(t), \ldots, u_{N_x,2}(t), \ldots,
u_{1,N_y}(t), \ldots, u_{N_x,N_y}(t) \right]^{\top}, \\
& f(\bm{u}(t)) = \left[f(u_{11}(t)), \ldots, f(u_{N_x,1}(t)), f(u_{12}(t)), \ldots, f(u_{N_x,2}(t)),
		\ldots, f(u_{1,N_y}(t)), \ldots, f(u_{N_x,N_y}(t)) \right]^{\top},
\end{align*}
$A = I_y \otimes A_x + A_y \otimes I_x$ with
\begin{equation*}
A_s = \frac{1}{h_s^{2}}
\begin{bmatrix}
-2 & 1 &  &  & 1 \\
1 & -2 & \ddots &  & \\
 & \ddots & \ddots & \ddots & \\
 &  & \ddots & -2 & 1 \\
1 &  &  & 1 & -2 \\
\end{bmatrix} \in \mathbb{R}^{N_s \times N_s}, \quad s = x, y.
\end{equation*}
Then, we obtain the following system of ordinary differential equations (ODEs):
\begin{equation}\label{eq2.1}
\frac{d \bm{u}(t)}{d t} = \tilde{A} \bm{u}(t) + A \bm{u}(t) + f(\bm{u}(t)),
\end{equation}
where $\tilde{A} = -\kappa A^2$.

We split Eq.~\eqref{eq2.1} into three parts.
This yields the following three subproblems:
\begin{equation}\label{eq2.2}
\frac{d \bm{v}(t)}{d t} = \tilde{A} \bm{v}(t), \qquad \bm{v}(0) = \bm{v}_0,
\end{equation}
\begin{equation}\label{eq2.3}
	\frac{d \bm{\varphi}(t)}{d t} = A \bm{\varphi}(t), \qquad \bm{\varphi}(0) = \bm{\varphi}_0
\end{equation}
and
\begin{equation}\label{eq2.4}
\frac{d \bm{\omega}(t)}{d t} = f(\bm{\omega}(t)), \qquad \bm{\omega}(0) = \bm{\omega}_0.
\end{equation}

Denote by $\Phi_{\tau}^{\tilde{A}}(\bm{v}_0)$, $\Phi_{\tau}^{A}(\bm{\varphi}_0)$ and
$\Phi_{\tau}^{f}(\bm{\omega}_0)$
the solutions of the subproblems \eqref{eq2.2}-\eqref{eq2.4} at $t_1$
with initial values $\bm{v}_0$, $\bm{\varphi}_0$ and $\bm{\omega}_0$, respectively.
Then, the FRS scheme with time step size $\tau$ is given by
\begin{equation}\label{eq2.5}
\mathcal{L}_{\tau} = \Phi_{\tau}^{A} \circ \Phi_{\tau}^{\tilde{A}} \circ \Phi_{\tau}^{f}.
\end{equation}
Set $\bm{\omega}_0 = \bm{u}(0)$,
we obtain the numerical solution of Eq.~\eqref{eq1.1} at $t_1$ as follows
\begin{equation*}
\bm{u}^1 = \mathcal{L}_{\tau} (\bm{u}(0)) =
\Phi_{\tau}^{A} \circ \Phi_{\tau}^{\tilde{A}} \circ \Phi_{\tau}^{f} (\bm{u}(0)).
\end{equation*}

Subsequently, the numerical solution of Eq.~\eqref{eq1.1} at $t_k$ is
$\bm{u}^k = \mathcal{L}_{\tau}^k (\bm{u}(0))$.
Moreover, the solution of \eqref{eq2.2} and \eqref{eq2.3} at $t_1$ can be written as
\begin{equation*}
\bm{v}(t_1) = e^{\tau \tilde{A}} \bm{v}_0 \quad \mathrm{and}
\quad \bm{\varphi}(t_1) = e^{\tau A} \bm{\varphi}_0,
\end{equation*}
respectively.
\subsection{Convergence and discrete maximum principle}
\label{sec2.2}

Let $\mathcal{V}_h = \left\{ v_h = v(x_h) \mid x_h \in \bar{\omega}_{h}~
\mathrm{and}~\bm{v}~\mathrm{is}~L\mathrm{\mbox{-}periodic~ in~ each~ direction} \right\}$.
A discrete inner product and the corresponding norm are defined as
\begin{equation*}
(\bm{u},\bm{v}) = h_x h_y \sum\limits_{x_h \in \bar{\omega}_{h}} u_h v_h, \quad
\| \bm{u} \| = \sqrt{(\bm{u}, \bm{u})}, \quad \forall \bm{u},\bm{v} \in \mathcal{V}_h.
\end{equation*}

For proving the convergence, the following lemma is needed.
\begin{lemma}(\cite{abbaszadeh2020error})\label{assumption1}
Suppose $u \in C^{\infty} (\Omega)$.
Then, the nonlinear term $f(u) = u - u^3$ in Eq.~\eqref{eq1.1} satisfies
\begin{equation*}
\left| f(u) - f(v) \right| < L \left| u - v \right|,
\end{equation*}
where $L$ is a Lipschitz constant.
\end{lemma}
Consequently, $f$ is continuous differentiable and bounded in a neighborhood of the solution $u$.
The matrices $A_x$ and $A_y$ are circulant matrices and can be diagonalized.
Then, we have the following property.
\begin{proposition}(\cite{van1992FFT}) \label{proposition1}
The matrices $A$ and $\tilde{A}$ in \eqref{eq2.1} can be diagonalized as
\begin{equation*}
A = \left( F_y^{*} \otimes F_x^{*} \right) \Lambda \left( F_y \otimes F_x \right), \quad
\tilde{A} = \left( F_y^{*} \otimes F_x^{*} \right) \tilde{\Lambda} \left( F_y \otimes F_x \right),
\end{equation*}
where $\tilde{\Lambda} = - \kappa \Lambda^2$.
The diagonal matrix
\begin{equation*}
\Lambda = I_y \otimes \Lambda_x + \Lambda_y \otimes I_x
\triangleq \mathrm{diag}(\lambda_1,\lambda_2,\ldots,\lambda_{N_x N_y})
\in \mathbb{R}^{N_x N_y \times N_x N_y}
\end{equation*}
contains all the eigenvalues of $A$
with
\begin{equation*}
	\Lambda_x = -\frac{4}{h_x^2}\, \mathrm{diag} \left( 0, \sin^2 \frac{\pi}{N_x}, \ldots,
	\sin^2 \frac{(N_x - 1)\pi}{N_x} \right) ~
\mathit{and} ~
	\Lambda_y = -\frac{4}{h_y^2}\, \mathrm{diag} \left( 0, \sin^2 \frac{\pi}{N_y}, \ldots,
	\sin^2 \frac{(N_y - 1)\pi}{N_y} \right).
\end{equation*}
Here, $F_s$ and $F_s^{*}$ ($s=x,y$) are discrete Fourier matrix and its conjugate transpose, respectively.
\end{proposition}

Obviously, $A$ and $\tilde{A}$ are negative semi-definite
and positive semi-definite matrices, respectively.
With these information at hand, we immediately get the following result.
\begin{lemma}\label{lemma2.1}
For $t > 0$, the matrix $B = \tilde{A} + A$ satisfies
\begin{itemize}
\item[(i)] {$\left\| e^{t B} \right\|_2 \leq 1$;}
\item[(ii)] {$\left\| e^{t B} (-B)^{\alpha} \right\|_2 \leq C_2 \; t^{-\alpha},\quad\alpha \geq 0$.}
\end{itemize}
Here, $\| \cdot \|_2$ is the spectral norm of a matrix and
the constant $C_2 = \max \left\{ 1, \alpha^{\alpha} e^{-\alpha} \right\}$
is independent of $t, h_x, h_y$ and the dimension of $B$.
\end{lemma}
\begin{proof}
(i) According to the decomposition in Proposition \ref{proposition1},
we know that all eigenvalues of $B$ are non-positive.
Then, we get
\begin{equation*}
\left\| e^{t B} \right\|_2 \leq 1.
\end{equation*}

(ii) If $\alpha = 0$, it just the inequality (i).
Next, we prove the result (ii) for $\alpha > 0 $.
We first give an upper bound of
$\left\| \left( -t B \right)^{\alpha} e^{t B} \right\|_2 (\alpha > 0)$.
Then, the target result is followed.
Denote by $\lambda_i^{B} = - \kappa \lambda_i^2 + \lambda_i$ ($i = 1, \ldots, N_x N_y$)
the $i$-th eigenvalue of $B$.
For $\alpha > 0 $, we have
\begin{equation*}
\left\| \left( -t B \right)^{\alpha} e^{t B} \right\|_2
= \left\| \mathrm{diag} \left( \left(-t \lambda_1^{B} \right)^{\alpha} e^{t \lambda_1^{B}},
\left(-t \lambda_2^{B} \right)^{\alpha} e^{t \lambda_2^{B}}, \ldots,
\left(-t \lambda_{N_x N_y}^{B} \right)^{\alpha} e^{t \lambda_{N_x N_y}^{B}} \right) \right\|_2.
\end{equation*}
Define an auxiliary function
\begin{equation*}
r(\theta) = (-\theta)^{\alpha} e^{\theta} \geq 0, \qquad \theta \leq 0, \alpha > 0.
\end{equation*}
It is easy to check that $\lim\limits_{\theta \rightarrow - \infty} r(\theta) = 0$.
The maximum of $r(\theta)$ is
\begin{equation*}
r(-\alpha) = \alpha^{\alpha} e^{-\alpha}.
\end{equation*}
Then, we obtain
\begin{equation*}
\left\| \left( -t B \right)^{\alpha} e^{t B} \right\|_2 \leq C_2, \quad t > 0, \alpha \geq 0.
\end{equation*}
The proof is completed.
\end{proof}

With the help of Lemmas \ref{assumption1} and \ref{lemma2.1},
we show the convergence of the FRS scheme \eqref{eq2.5}.
\begin{theorem}\label{th2.1}
Assume $u \in C^{\infty} (\Omega)$.
Temporal convergence order of the FRS scheme \eqref{eq2.5} is $1$ and the global error satisfies
\begin{equation*}
\left\| \mathcal{U}(t_k) - \bm{u}^k \right\|
\leq C \tau \left( 1 + | \log \tau | \right) + \tilde{C} \left( h_x^2 + h_y^2 \right), \quad 0 \leq k \tau \leq T,
\end{equation*}
where
\begin{equation*}
\mathcal{U}(t_k) = \left[u(x_1, y_1,t_k), \ldots, u(x_{N_x},y_1,t_k), \ldots,
u(x_1, y_{N_y},t_k), \ldots, u(x_{N_x},y_{N_y},t_k) \right]^{\top},
\end{equation*}
the constants $C$ and $\tilde{C}$ are independent of $\tau$ and $k$.
\end{theorem}
\begin{proof}
We notice that
\begin{equation*}
\left\| \mathcal{U}(t_k) - \bm{u}^k \right\|
\leq \left\| \bm{u}(t_k) - \bm{u}^k \right\| + \left\| \bm{u}(t_k) - \bm{u}^k \right\|
\leq \left\| \bm{u}(t_k) - \bm{u}^k \right\| + \tilde{C} \left( h_x^2 + h_y^2 \right).
\end{equation*}
Thus, the remaining is to estimate $\left\| \bm{u}(t_k) - \bm{u}^k \right\|$.
For the FRS scheme \eqref{eq2.5} with initial value $\bm{u}^{k - 1}$, it gives
\begin{equation*}
\begin{split}
\bm{u}^k = \mathcal{L}_{\tau} \bm{u}^{k - 1} & =
\Phi_{\tau}^{A} \left( \Phi_{\tau}^{\tilde{A}} \left( \Phi_{\tau}^{f}
\left( \bm{u}^{k - 1} \right) \right) \right) \\
& = e^{\tau (A + \tilde{A})} \bm{u}^{k - 1} + \tau e^{\tau (A + \tilde{A})} f(t_{k - 1}, \bm{u}^{k - 1}) +
\int_{0}^{\tau} e^{\tau (A + \tilde{A})} (\tau - s) {\omega}''(t_{k - 1} + s)\, ds.
\end{split}
\end{equation*}
On the other hand, the exact solution of \eqref{eq2.1} with initial value $\mathcal{U}(t_{k - 1})$
can be written as
\begin{equation*}
\begin{split}
\mathcal{U}(t_k) = e^{\tau (A + \tilde{A})} \mathcal{U}(t_{k - 1}) +
\int_{0}^{\tau} e^{(\tau - s) (A + \tilde{A})} f \left( t_{k - 1}, \mathcal{U}(t_{k - 1} + s) \right)\, ds.
\end{split}
\end{equation*}
Then, we have the local error
\begin{equation*}
\bm{u}^k - \mathcal{U}(t_k) = \int_{0}^{\tau} e^{\tau (A + \tilde{A})} (\tau - s)
{\omega}''(t_{k - 1} + s)\, ds -
\int_{0}^{\tau} \int_{0}^{s} {\ell}'_k (\xi)\, d\xi ds,
\end{equation*}
where
\begin{equation*}
\ell_k (\xi) = e^{(\tau - \xi) (A + \tilde{A})} f \left( t_{k - 1} + \xi, \mathcal{U}(t_{k - 1} + \xi) \right).
\end{equation*}

The remaining proof is very similar to \cite{ostermann2019convergence,einkemmer2015overcoming}.
Thus, we omit it here. The proof is completed.
\end{proof}

At the end of this section, we show that the FRS scheme \eqref{eq2.5}
preserves the discrete maximum principle unconditionally.
\begin{theorem}\label{th2.2}
Suppose the initial value satisfies $\left\| \bm{u}^0 \right\|_{\infty} \leq 1$.
Then, the FRS scheme \eqref{eq2.5} preserves the discrete maximum principle for any $\tau > 0$,
i.e.,
\begin{equation*}
\left\| \bm{u}^k \right\|_{\infty} \leq 1, \quad \mathit{for~all~} k \geq 0,
\end{equation*}
where $\left\| \cdot \right\|_{\infty}$ is the standard infinity norm for a vector or a matrix.
\end{theorem}
\begin{proof}
We proceed by induction. Under the assumption, it is true for $k = 0$.
Assume the result holds for $k \leq m$, that is, $\left\| \bm{u}^m \right\|_{\infty} \leq 1$
($k = 1,\ldots,m$).

Next, we check the inequality holds for $k = m + 1$.
We know that for a matrix $Q$, the relation
$\left\| Q \right\|_{\infty} \leq \left\| Q \right\|_{2}$ is true \cite{quarteroni2010numerical}.
Noticing Lemma \ref{lemma2.1}, we obtain
\begin{equation*}
\left\| e^{\tau B} \right\|_{\infty} \leq \left\| e^{\tau B} \right\|_{2} \leq 1.
\end{equation*}
Using Lemma 3.4 in \cite{chen2022second}, we have
\begin{equation*}
\left\| \bm{u}^{m + 1} \right\|_{\infty} =
\left\| \Phi_{\tau}^{A} \circ \Phi_{\tau}^{\tilde{A}} \circ \Phi_{\tau}^{f} (\bm{u}^{m}) \right\|_{\infty}
= \left\| e^{\tau B} \left( \Phi_{\tau}^{f} (\bm{u}^{m}) \right) \right\|_{\infty}
\leq \left\| \Phi_{\tau}^{f} (\bm{u}^{m}) \right\|_{\infty}
\leq 1.
\end{equation*}
This completes the proof of the induction.
\end{proof}

\section{A rank-adaptive low-rank splitting approach}
\label{sec3}

In this section,
we propose a rank-adaptive low-rank splitting (ALRS) approach for Eq.~\eqref{eq1.1}.
We first rewrite subproblems \eqref{eq2.4}, \eqref{eq2.2} and \eqref{eq2.3}
into the equivalent matrix differential equations:
\begin{equation}\label{eq3.1}
\dot{W}(t) = f(W(t)) = W(t) - W(t).^3, \qquad W(0) = W_0,
\end{equation}
\begin{equation}\label{eq3.2}
	\dot{K}(t) = \tilde{A} \circledast K(t), \quad K(t_0) = K_0
\end{equation}
and
\begin{equation}\label{eq3.3}
	\dot{\Psi}(t) = A_x \Psi(t) + \Psi(t) A_y, \quad \Psi(t_0) = \Psi_0,
\end{equation}
respectively.
Here $\dot{W}(t) = \frac{d W(t)}{d t}$, $\dot{K}(t) = \frac{d K(t)}{d t}$,
$\dot{\Psi}(t) = \frac{d \Psi(t)}{d t}$ and $.^3$ is elementwise operation.
It is difficult to rewrite the subproblem \eqref{eq2.2} into a matrix form directly,
since the matrix $\tilde{A}$ is hard to be decoupled.
Thanks to the decomposition of $\tilde{A}$, we can transform the matrix-vector multiplication
(i.e., $\tilde{A} v$, $v$ is a vector) to a new matrix-matrix operation,
that is, ``$\circledast$", cf.~Algorithm \ref{alg1}.
With this new operation, the subproblem \eqref{eq2.2} will be equivalent to \eqref{eq3.2}.
The time-dependent matrix $W(t) \in \mathbb{R}^{N_x \times N_y}$ is obtained
by reshaping the vector $\bm{\omega}(t)$ into a matrix,
i.e., $W(t) = \mathrm{reshape}(\bm{\omega}(t),N_x,N_y)$.
Similarly, $K(t) = \mathrm{reshape}(\bm{v}(t),N_x,N_y)$ and
$\Psi(t) = \mathrm{reshape}(\bm{\varphi}(t),N_x,N_y)$.
\begin{algorithm}[ht]
	\caption{Compute $Z = \tilde{A} \circledast H$ }
	\begin{algorithmic}[1]
		\STATE {Given a matrix $H \in \mathbb{R}^{N_x \times N_y}$}
		\STATE {Compute
			\begin{align*}
				\Lambda_x = -\frac{4}{h_x^2} \left[ 0, \sin^2 \frac{\pi}{N_x}, \sin^2 \frac{2 \pi}{N_x}, \ldots, \sin^2 \frac{(N_x - 1) \pi}{N_x} \right]^T, \\
				\Lambda_y = -\frac{4}{h_y^2} \left[ 0, \sin^2 \frac{\pi}{N_y}, \sin^2 \frac{2 \pi}{N_y},
				\ldots, \sin^2 \frac{(N_y - 1) \pi}{N_y} \right]^T
			\end{align*}
			and
			$\tilde{\Lambda} = -\kappa \left( \bm{1}_{N_x,N_y} + \Lambda_x\; \bm{1}_{1,N_y} + \bm{1}_{N_x,1}\; \Lambda_y^T \right).^ 2$ \quad ($.^2$ is elementwise operation.)}
		\STATE {Compute $V_1 = \mathtt{fft2}(H)$ \quad ($\mathtt{fft2}$ represents 2D discrete Fourier transform.)}
		\STATE {Compute $V_2 = \tilde{\Lambda} \bullet V_1$ \quad ($\bullet$ means Hadamard product.)}
		\STATE {Compute $Z = \mathtt{real} \left( \mathtt{ifft2}(V_2) \right)$ \quad ($\mathtt{ifft2}$ represents the inverse of $\mathtt{fft2}$, $\mathtt{real}$ the real part of a complex number.)}
	\end{algorithmic}
	\label{alg1}
\end{algorithm}

Denote by $\mathcal{M}_r = \left\{ Q(t) \in \mathbb{R}^{N_x \times N_y} \mid \mathrm{rank}(Q(t)) = r\right\}$
the manifold of rank-$r$ matrices.
Let $\mathcal{T}_{\tilde{W}(t)} \mathcal{M}_r$ be the tangent space of $\mathcal{M}_r$ at $\tilde{W}(t)$.
According to \cite{koch2007dynamical}, a low-rank numerical solution of \eqref{eq3.1} is
obtained by solving the following optimization problem
\begin{equation*}
\min_{\tilde{W}(t) \in \mathcal{M}_r} \left\| \dot{\tilde{W}}(t) - \dot{W}(t) \right\|_{F},\quad
\mathrm{s.t.}~ \dot{\tilde{W}}(t) \in \mathcal{T}_{\tilde{W}(t)} \mathcal{M}_r,
\end{equation*}
where $\| \cdot \|_F$ represents the Frobenius norm.
Clearly, this optimization problem is equivalent to the following evolution equation
\begin{equation}\label{eq3.4}
\dot{\tilde{W}}(t) = P(\tilde{W}(t))\, f(\tilde{W}(t)), \quad \tilde{W}(t_0) = W_0 \in \mathcal{M}_r,
\end{equation}
where $P(\tilde{W}(t))$ is the orthogonal projection onto $\mathcal{T}_{\tilde{W}(t)} \mathcal{M}_r$.
This equation can be solved by employing projector-splitting integrators \cite{lubich2014projector, Ceruti2022unconventional} or the low-rank Parareal \cite{Carrel2022DLRparareal}.
Considering the optimal rank required for approximating Eq.~\eqref{eq3.1}
may vary strongly over time.
Thus, we choose the rank-adaptive integrator \cite{Ceruti2022adaptiveDLR} for solving \eqref{eq3.4}.
Moreover, denote by $\Phi_{\tau,r}^f (W_0)$ the low-rank approach of \eqref{eq3.1}.

Notice that \eqref{eq3.2} is not rank preserving.
On the other hand, it becomes a stiff problem, if $\kappa$ is large enough.
If the DLRA is used for the subproblem \eqref{eq3.2}, some stiff ODE solvers
should be considered in each substep of the rank-adaptive integrator \cite{Ceruti2022adaptiveDLR}.
In this work, we choose another way to find a low-rank numerical solution of \eqref{eq3.2}.
More precisely, a low-rank numerical solution of \eqref{eq3.2} with initial value $K_0$
through the following optimization problem:
\begin{equation*}
	\min_{\tilde{K} \in \mathcal{M}_r} \left\| \tilde{K} -
	e^{\tau \tilde{A}} \circledast K_0 \right\|_{F}.
\end{equation*}
Obviously, this optimization problem can be solved by employing a truncated singular value decomposition.
Furthermore, we denote the low-rank solution of \eqref{eq3.2} as $\Phi_{\tau,r}^{\tilde{A}} (K_0)$
($K_0 \in \mathcal{M}_r$).

For the subproblem \eqref{eq3.3}, it is rank preserving \cite[Lemma 1.22]{helmke1996optimization}.
This means that for a given rank-$r$ initial value $Y_0$,
the solution of
\begin{equation*}
\dot{Y}(t) = A_x Y(t) + Y(t) A_y, \quad Y(t_0) = Y_0
\end{equation*}
is still rank-$r$ for all $t$.
The exact solution of this equation at $t_1$ is
\begin{equation*}
Y(t_1) = e^{\tau A_x}\, Y_0\, e^{\tau A_y}.
\end{equation*}
This expression is the same as the exact solution of \eqref{eq3.3}.
The only difference is the initial value.
Thus, the low-rank solution of \eqref{eq3.3} is denoted as $\Phi_{\tau}^{A} (\Psi_0)$
($\Psi_0 \in \mathcal{M}_r$).

With the help of these notations, our ALRS procedure is given as follows
\begin{equation}\label{eq3.5}
\mathcal{L}_{\tau,r} = \Phi_{\tau}^{A} \circ \Phi_{\tau,r}^{\tilde{A}} \circ \Phi_{\tau,r}^{f}.
\end{equation}

Let $U^0$ be a rank-$r_0$ approximation of the initial value $u(x,y,0)$.
Starting with $W_0 = U^0$, we obtain the low-rank solution of \eqref{eq1.1} at $t_k$
\begin{equation*}
U^k = \mathcal{L}_{\tau,r}^k (U^0).
\end{equation*}

\section{Numerical experiments}
\label{sec4}

In this section, two examples are provided to show the performance of our proposed methods.
Example 1 shows the observed convergence orders of \eqref{eq2.5} and \eqref{eq3.5}.
Example 2 simulates mean curvature effect motion for various initial shapes.
In these examples, we fix $N_x = N_y = N$.
For the rank-adaptive integrator \cite{Ceruti2022adaptiveDLR},
we set the truncation tolerance $\vartheta = 10^{-3}$.
Some notations that will appear later are collected here:
\begin{equation*}
	Err_\infty(h, \tau) = \max_{0 \leq i \leq N} \mid u(x_i,T) - u_{i}^{M} \mid,~
	Err_2(h, \tau) = \| \mathcal{U}(T) - \bm{u}^M \|,~
	\textrm{relerr}(\tau,h) =
	\frac{\left\| U^M - \mathcal{U}(T) \right\|_{F}}{\left\| \mathcal{U}(T) \right\|_{F}},
\end{equation*}
\begin{equation*}
	CO_{\infty, \tau} =\log_{\tau_1/ \tau_2} \frac{Err_\infty(h, \tau_1)}{Err_\infty(h, \tau_2)}, \quad
	CO_{2, \tau} =\log_{\tau_1/ \tau_2} \frac{Err_2(h, \tau_1)}{Err_2(h, \tau_2)},
\end{equation*}
\begin{equation*}
	CO_{\infty, h} =\log_{h_1/ h_2} \frac{Err_\infty(h_1, \tau)}{Err_\infty(h_2, \tau)}, \quad
	CO_{2, h} =\log_{h_1/ h_2} \frac{Err_2(h_1, \tau)}{Err_2(h_2, \tau)},
\end{equation*}
\begin{equation*}
	\textrm{rate}_{\tau} = \log_{\tau_1/\tau_2} \frac{\textrm{relerr}(\tau_1,h)}{\textrm{relerr}(\tau_2,h)}
	\quad \mathrm{and} \quad
	\textrm{rate}_{h} = \log_{h_1/h_2} \frac{\textrm{relerr}(\tau,h_1)}{\textrm{relerr}(\tau,h_2)}.
\end{equation*}

All experiments were performed on a Windows 10 (64 bit) PC-Intel(R) Core(TM) 1135G7
CPU 2.40 GHz, 24 GB of RAM using MATLAB R2018b.

\noindent\textbf{Example 1.}~(Convergence test) Considering Eq.~\eqref{eq1.1} with $\kappa = 0.01$,
$\Omega = [0,32]^2$, $T = 1$ and
\begin{equation*}
\begin{split}
u_0(x,y) = & 0.1 - 0.2 \cos \left( \frac{2 \pi (x - 12)}{32} \right)\,
\sin \left( \frac{2 \pi (y - 1)}{32} \right) +
0.1 \cos^2 \left( \frac{\pi (x + 10)}{32} \right)\, \sin^2 \left( \frac{\pi (y + 3)}{32} \right) - \\
& 0.2 \sin^2 \left( \frac{4 \pi x}{32} \right)\, \cos \left( \frac{4 \pi (y - 6)}{32} \right).
\end{split}
\end{equation*}
The exact solution is unknown.
Thus, the numerical solution on the fine mesh ($M = 2048, N = 1024$) computed by the FRS scheme \eqref{eq2.5}
is treated as the reference solution.

Tables \ref{tab1} and \ref{tab2} report the observed time and
space convergence orders of the FRS scheme \eqref{eq2.5}.
From these tables, convergence orders of \eqref{eq2.5} in time and space indeed $1$ and $2$,
respectively. This is consistent with our theoretical analysis.
For the ALRS \eqref{eq3.5}, the relative errors and observed convergence orders
are listed in Tables \ref{tab3} and \ref{tab4}.
As we can see in these tables, for a small initial rank (i.e., $r_0 = 3$),
a stagnation of the error is observed.
The reason for this is that a too low rank produces numerical solutions of low quality.
For the case $r_0 = 4$, the observed time and space convergence orders is $1$ and $2$,
respectively. This coincides with our expectations.
Fig.~\ref{fig1} show the maximum norm and the energy of the numerical solution
computed by \eqref{eq2.5} and \eqref{eq3.5}.
We see that our methods can preserve the maximum principle and energy decaying.
Fig.~\ref{fig2} compares the rank of the numerical solution
computed by \eqref{eq2.5} and \eqref{eq3.5}.
We observe that the effective rank of the solution stays low during its evolution in time.
The rank of the low-rank solution is lower than that of the solution computed by \eqref{eq2.5}.
\begin{table}[ht]\tabcolsep=8pt
	\caption{Numerical errors and observed time convergence orders of \eqref{eq2.5} for Example 1
		with $N = 1024$.}
	\centering
	\begin{tabular}{ccccc}
		\hline
		$M$ & $Err_\infty(h,\tau)$ & $CO_{\infty, \tau}$ & $Err_2(h,\tau)$ & $CO_{2, \tau}$ \\
		\hline
16 & 1.4104E-03 & -- & 1.1510E-02 & -- \\	
32 & 6.9556E-04 & 1.0199 & 5.6540E-03 & 1.0256 \\
64 & 3.4124E-04 & 1.0274 & 2.7685E-03 & 1.0302 \\
128 & 1.6487E-04 & 1.0495 & 1.3363E-03 & 1.0508 \\
256 & 7.6881E-05 & 1.1006 & 6.2285E-04 & 1.1013 \\
		\hline
	\end{tabular}
	\label{tab1}
\end{table}
\begin{table}[ht]\tabcolsep=8pt
	\caption{Numerical errors and observed space convergence orders of \eqref{eq2.5} for Example 1
		with $M = 2048$.}
	\centering
	\begin{tabular}{ccccc}
		\hline
		$N$ & $Err_\infty(h,\tau)$ & $CO_{\infty, h}$ & $Err_2(h,\tau)$ & $CO_{2, h}$ \\
		\hline
		16 & 1.9370E-02 & -- & 2.4071E-01 & -- \\	
		32 & 4.9034E-03 & 1.9819 & 5.8648E-02 & 2.0372 \\
		64 & 1.2249E-03 & 2.0012 & 1.4358E-02 & 2.0302 \\
		128 & 3.0273E-04 & 2.0165 & 3.5132E-03 & 2.0310 \\
		256 & 7.2084E-05 & 2.0703 & 8.3234E-04 & 2.0775 \\
		\hline
	\end{tabular}
	\label{tab2}
\end{table}
\begin{table}[H]\tabcolsep=8pt
	\caption{Relative errors in Frobenius norm
		and observed time convergence orders of \eqref{eq3.5} for $N = 1024$ for Example 1.}
	\centering
	\begin{tabular}{ccccc}
		\hline
		& \multicolumn{2}{c}{$r_0 = 3$} & \multicolumn{2}{c}{$r_0 = 4$} \\
		[-2pt] \cmidrule(lr){2-3} \cmidrule(lr){4-5} \\ [-11pt]
		$M$ & $\textrm{relerr}(\tau,h)$ & $\textrm{rate}_\tau$
		& $\textrm{relerr}(\tau,h)$ & $\textrm{rate}_\tau$ \\
		\hline
16 & 3.2166E-02 & -- & 9.4197E-04 & -- \\
32 & 3.2153E-02 & 0.0006 & 4.6241E-04 & 1.0265 \\
64 & 3.2150E-02 & 0.0001 & 2.2650E-04 & 1.0296 \\
128 & 3.2150E-02 & 0.0000 & 1.0951E-04 & 1.0484 \\
256 & 3.2149E-02 & 0.0000 & 5.1281E-05 & 1.0946 \\
		\hline
	\end{tabular}
	\label{tab3}
\end{table}
\begin{table}[ht]\tabcolsep=8pt
	\caption{Relative errors in Frobenius norm
		and observed space convergence orders of \eqref{eq3.5} for $M = 2048$ for Example 1.}
	\centering
	\begin{tabular}{ccccc}
		\hline
		& \multicolumn{2}{c}{$r_0 = 3$} & \multicolumn{2}{c}{$r_0 = 4$} \\
		[-2pt] \cmidrule(lr){2-3} \cmidrule(lr){4-5} \\ [-11pt]
		$N$ & $\textrm{relerr}(\tau,h)$ & $\textrm{rate}_h$
		& $\textrm{relerr}(\tau,h)$ & $\textrm{rate}_h$ \\
		\hline
		16 & 3.8944E-02 & -- & 1.7878E-02 & -- \\
		32 & 3.3362E-02 & 0.2232 & 4.4831E-03 & 1.9956 \\
		64 & 3.2572E-02 & 0.0346 & 1.1147E-03 & 2.0079 \\
		128 & 3.2333E-02 & 0.0106 & 2.7614E-04 & 2.0131 \\
		256 & 3.2225E-02 & 0.0048 & 6.6644E-05 & 2.0509 \\
		\hline
	\end{tabular}
	\label{tab4}
\end{table}
\begin{figure}[H]
	\centering
	\subfigure
	{\includegraphics[width=3.0in,height=2.4in]{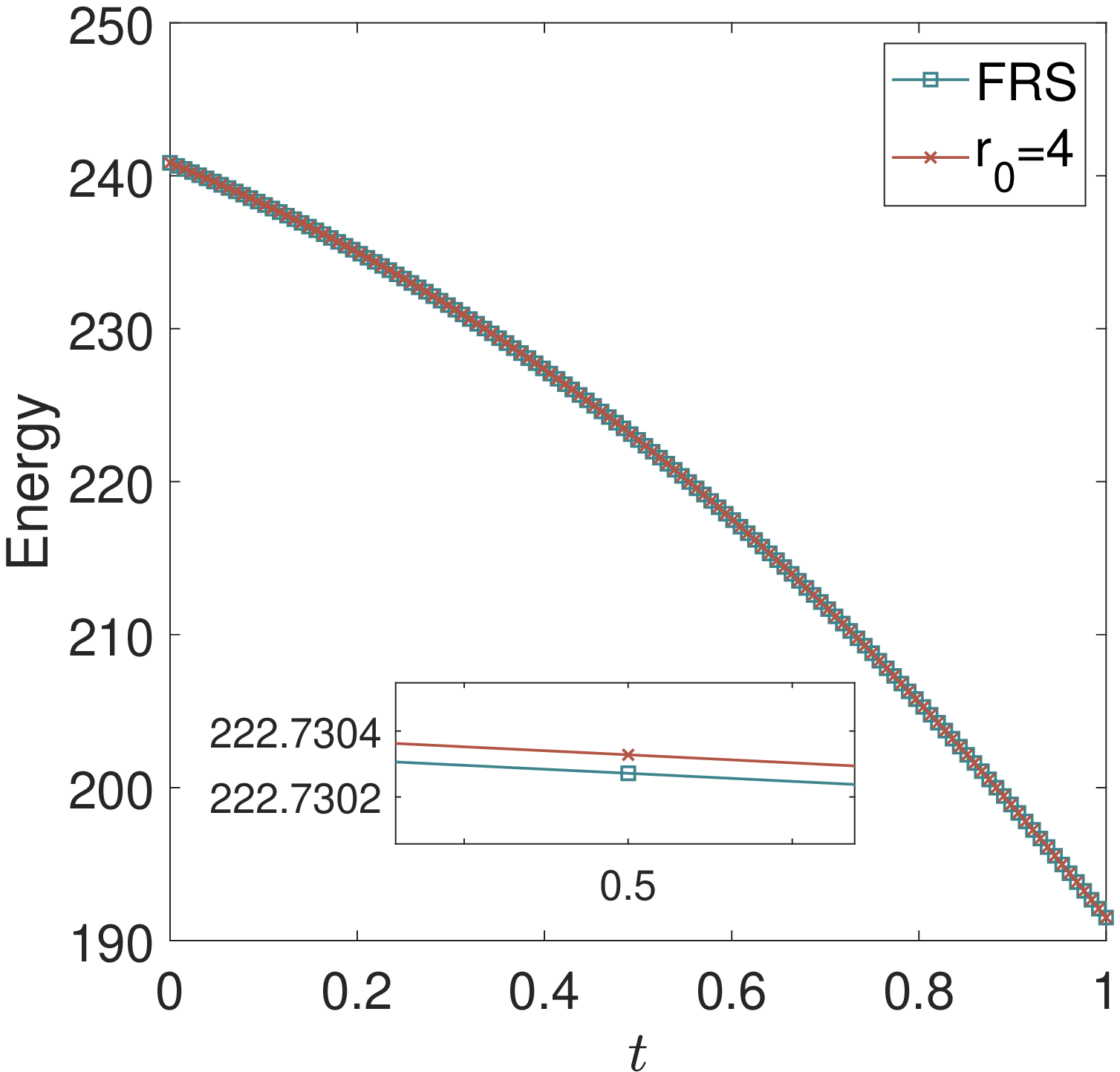}} \hspace{2mm}
	\subfigure
	{\includegraphics[width=3.0in,height=2.4in]{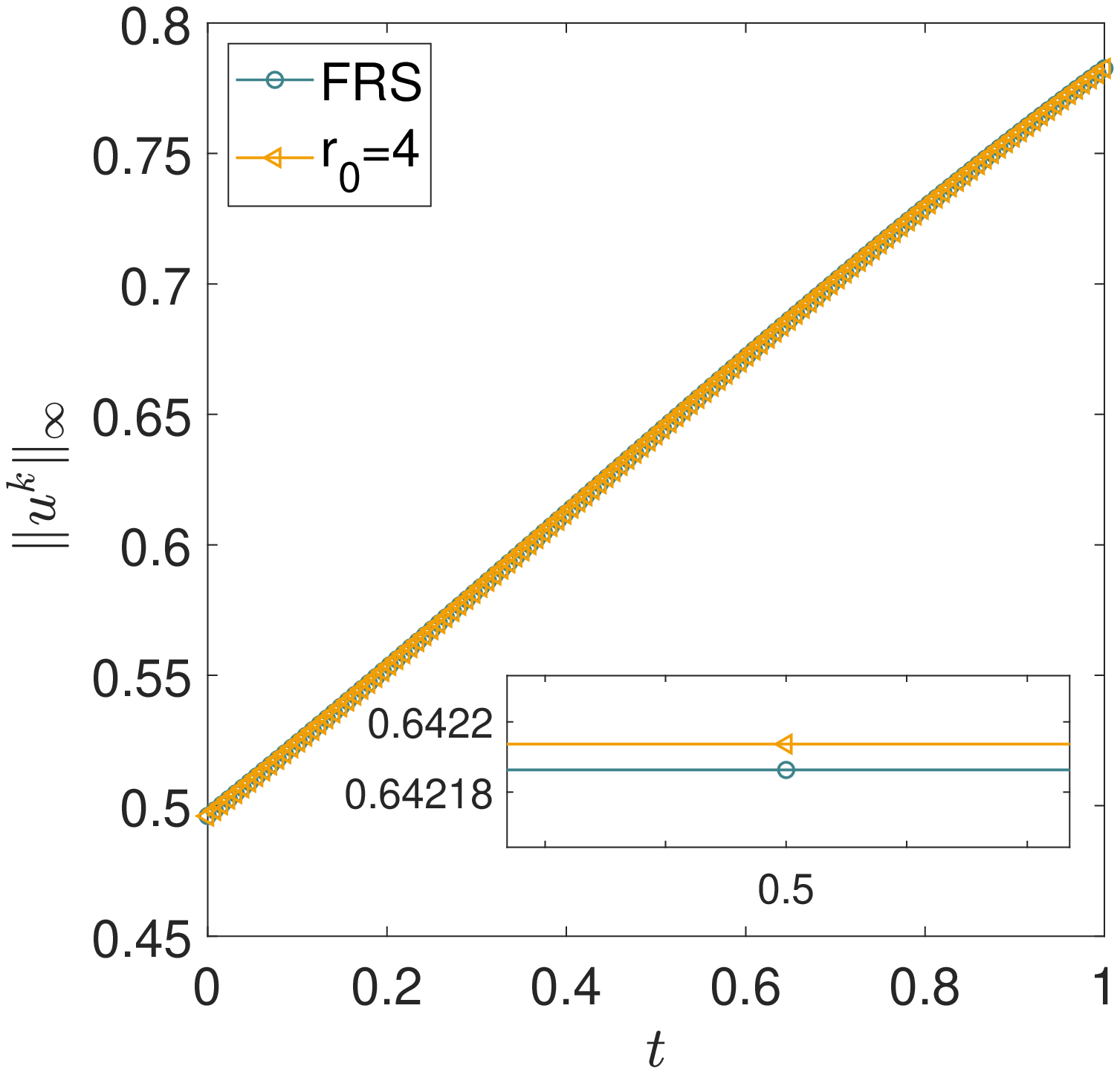}}
	\caption{The maximum norm (left) and the energy (right) of the numerical solution
		computed by \eqref{eq2.5} and \eqref{eq3.5} for Example 1 with $(M,N) = (128,1024)$.}
	\label{fig1}
\end{figure}
\begin{figure}[ht]
	\centering
	{\includegraphics[width=3.0in,height=2.4in]{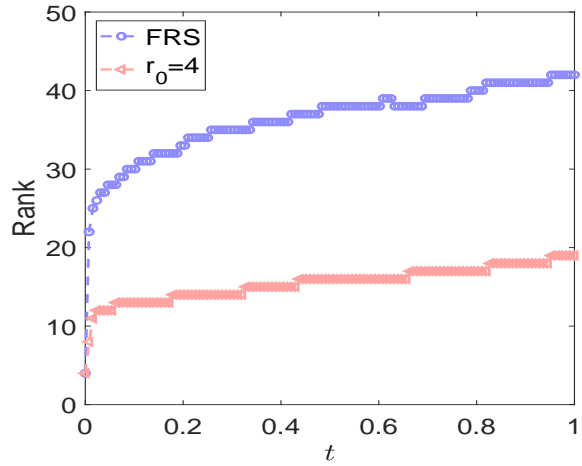}}
	\caption{The rank comparison of the numerical solution
		computed by \eqref{eq2.5} and \eqref{eq3.5} for Example 1 with $(M,N) = (128,1024)$.}
	\label{fig2}
\end{figure}

\noindent\textbf{Example 2.}~(Numerical simulation) In this example,
we consider Eq.~\eqref{eq1.1} with $\kappa = 10^{-4}$ and the following three different initial shapes \cite{kim2022learning}:

(1) Star: $u(x,y,0) = \tanh \left( \frac{0.25 + 0.1 \cos(6 \theta) -
			\sqrt{(x - 0.5)^2 + (y - 0.5^2)}}{\varepsilon \sqrt{2}} \right)$,
where $(x,y) \in [0,1]^2$ and
\begin{equation*}
\theta =
\begin{cases}
\tan^{-1} \left( \frac{y - 0.5}{x - 0.5} \right), & x > 0.5, \\
\pi + \tan^{-1} \left( \frac{y - 0.5}{x - 0.5} \right), & \mathrm{others}.
\end{cases}	
\end{equation*}
\begin{figure}[p]
	\setlength{\tabcolsep}{0.2pt}
	\centering
	\begin{tabular}{m{0.4cm}<{\centering} m{4cm}<{\centering} m{4cm}<{\centering} m{4cm}<{\centering} m{4cm}<{\centering}}
		& $t = 0$ & $t = 0.0016$ & $t = 0.005$ & $t = 0.01$ \\
		\rotatebox{90}{FRS} &
		\includegraphics[width=1.6in,height=1.2in]{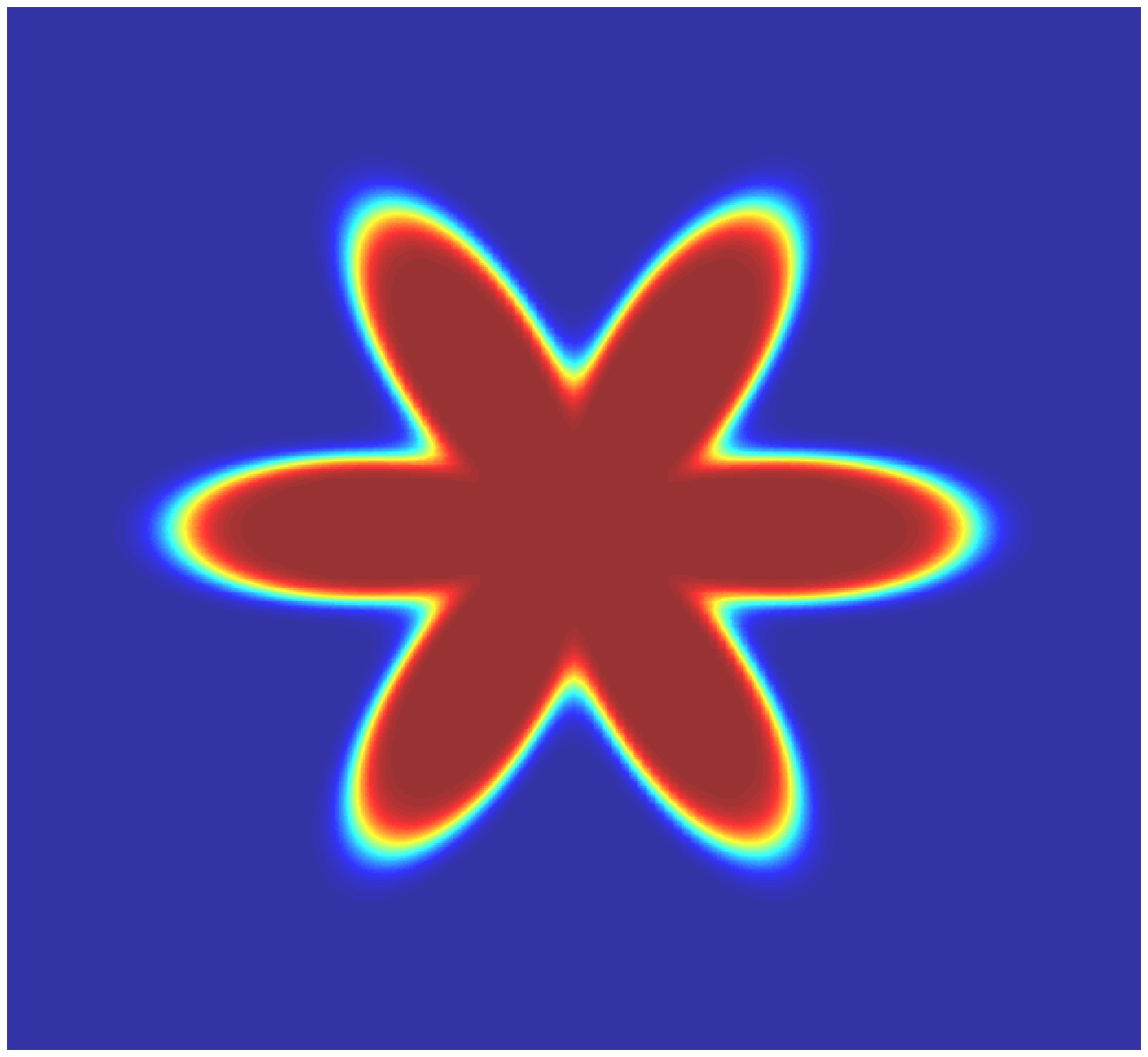} &
		\includegraphics[width=1.6in,height=1.2in]{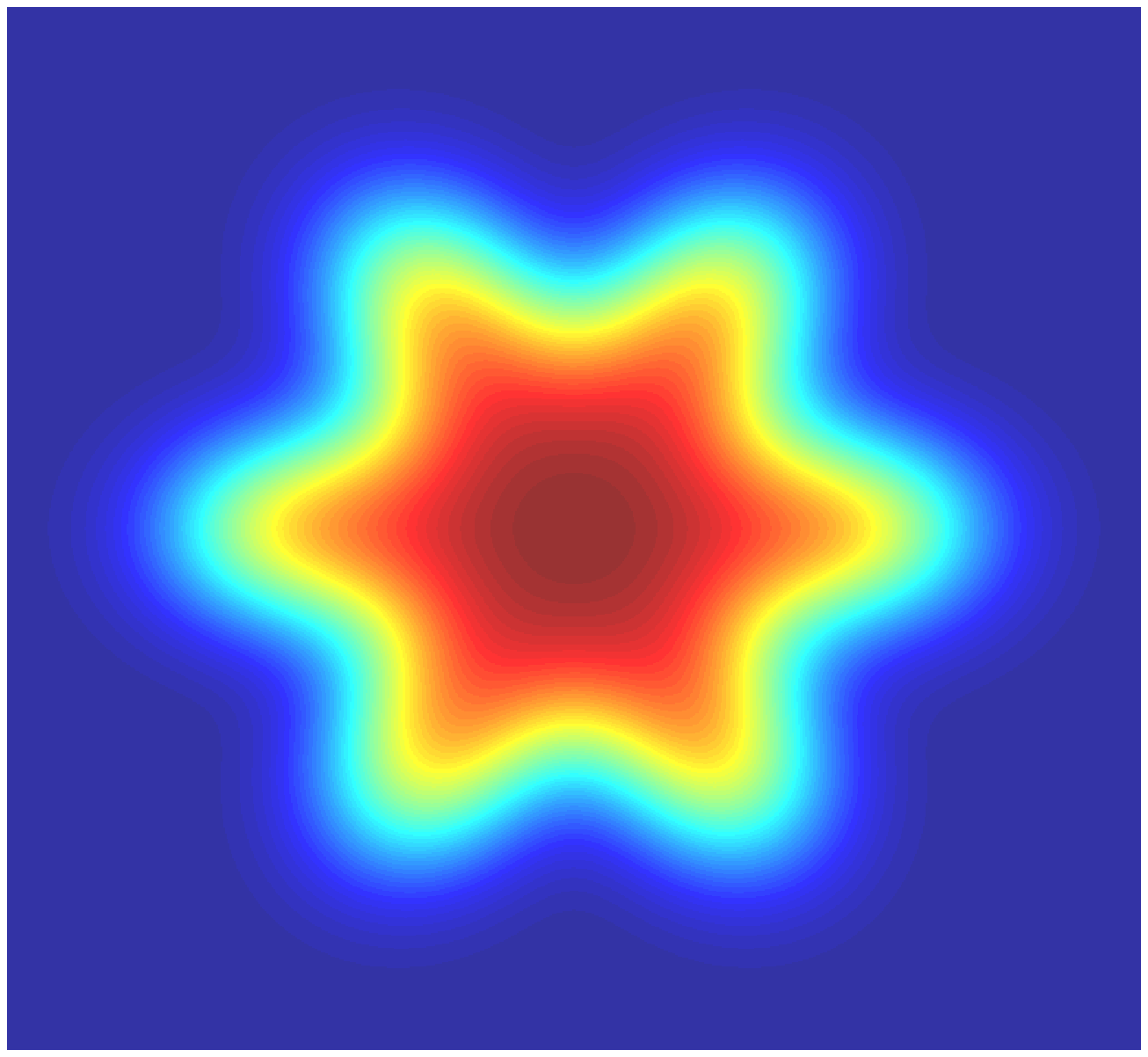} &
		\includegraphics[width=1.6in,height=1.2in]{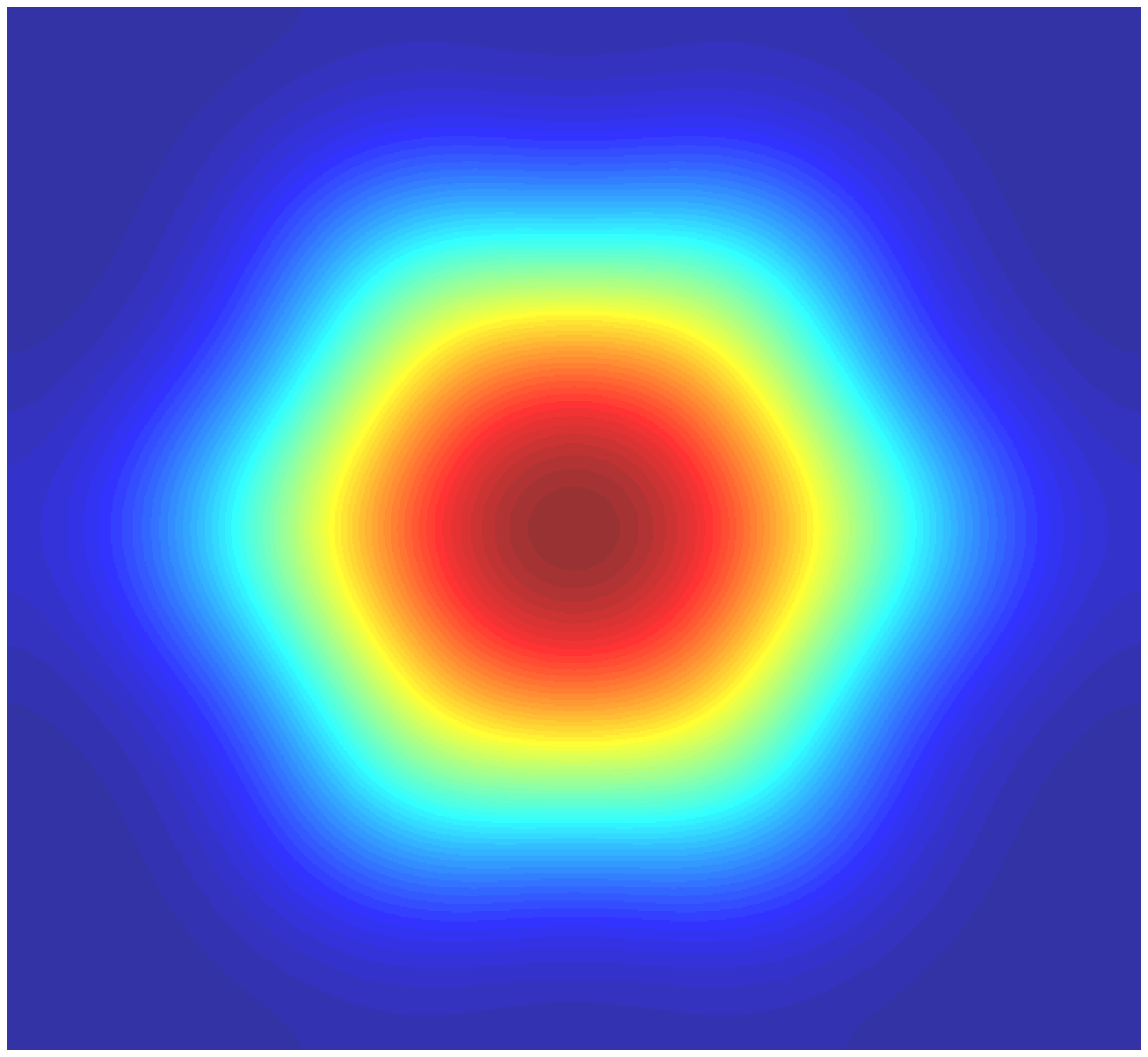} &
		\includegraphics[width=1.6in,height=1.2in]{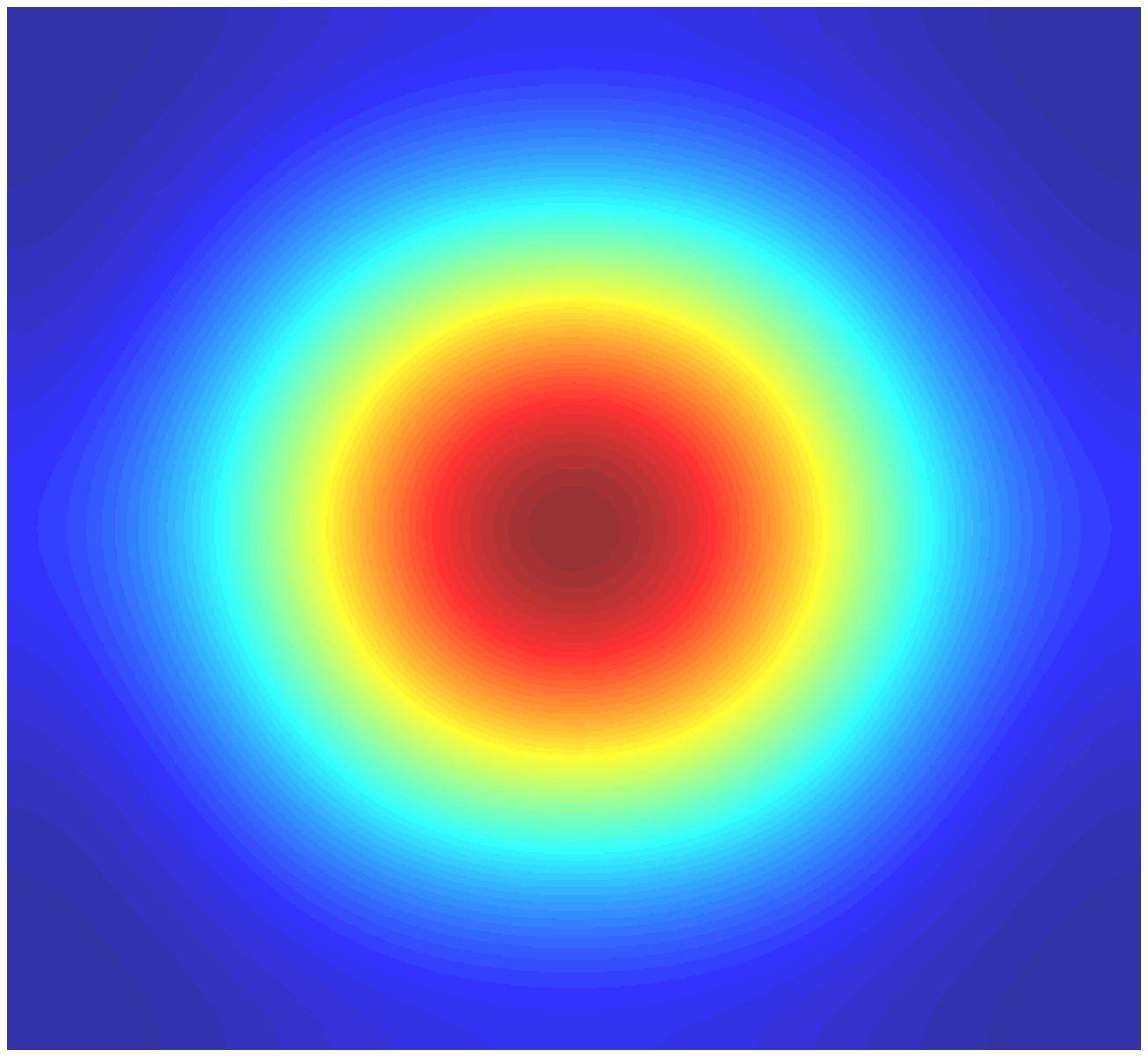} \\
		\rotatebox{90}{$r_0 = 16$} &
		\includegraphics[width=1.6in,height=1.2in]{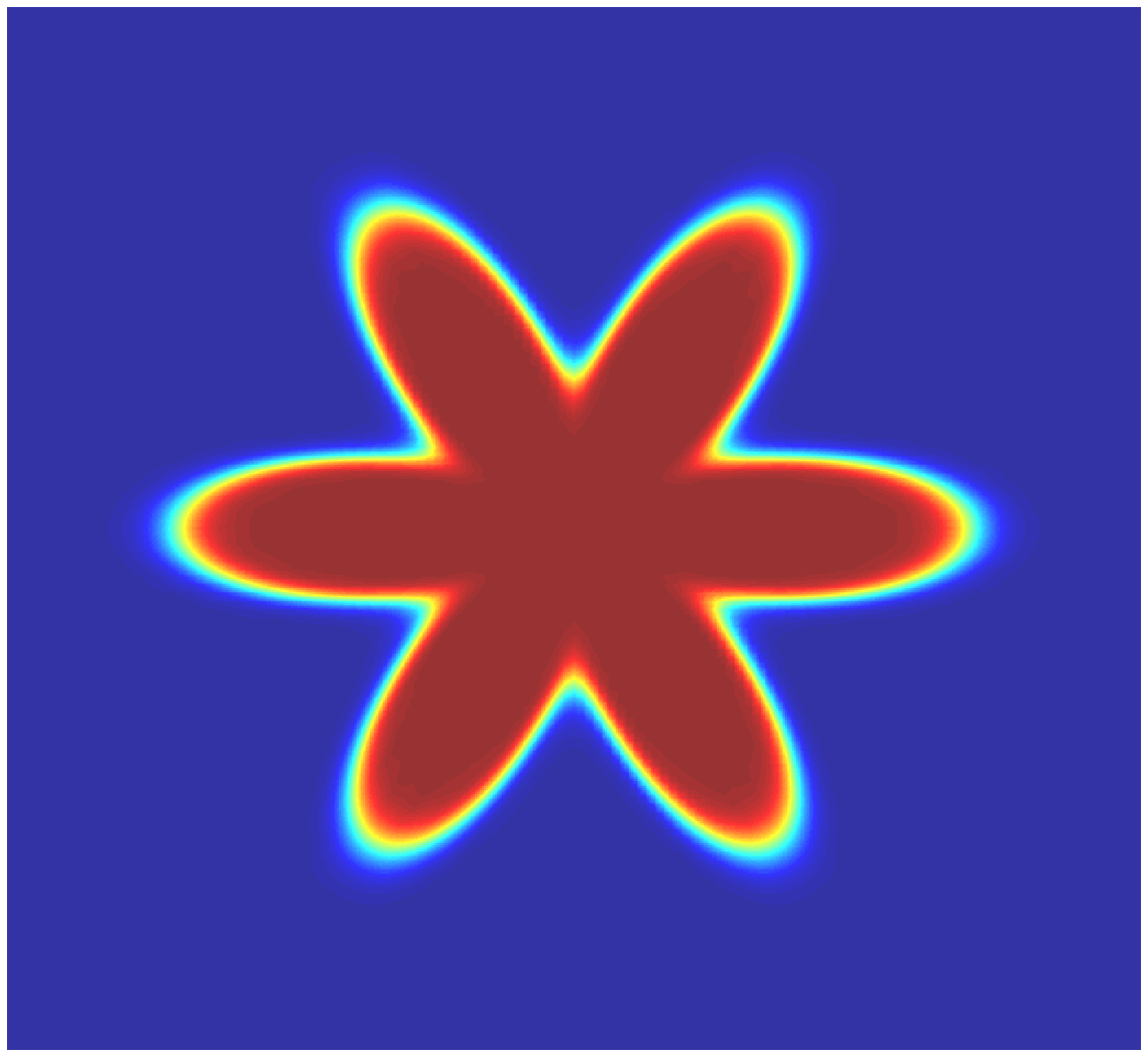} &
		\includegraphics[width=1.6in,height=1.2in]{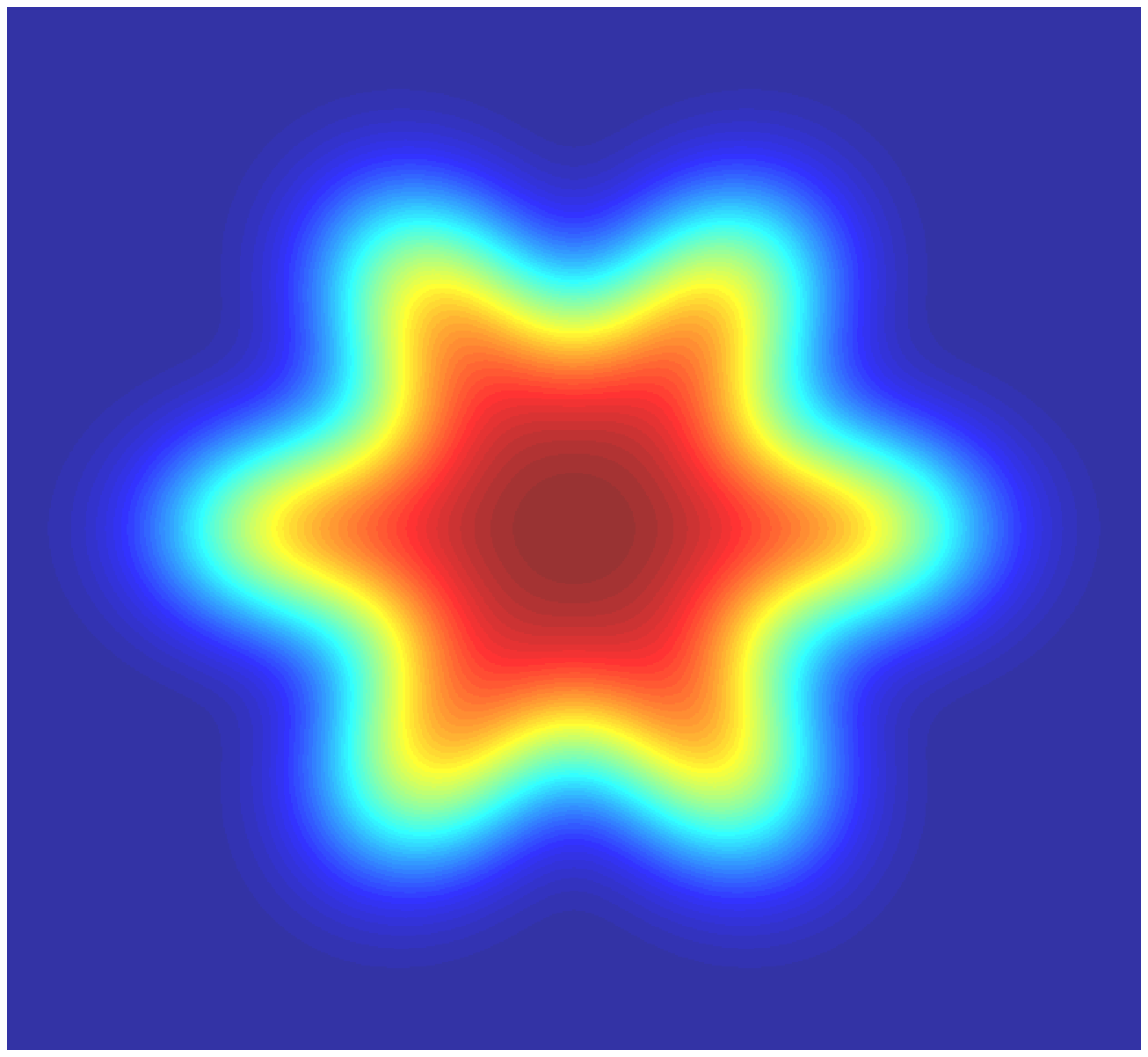} &
		\includegraphics[width=1.6in,height=1.2in]{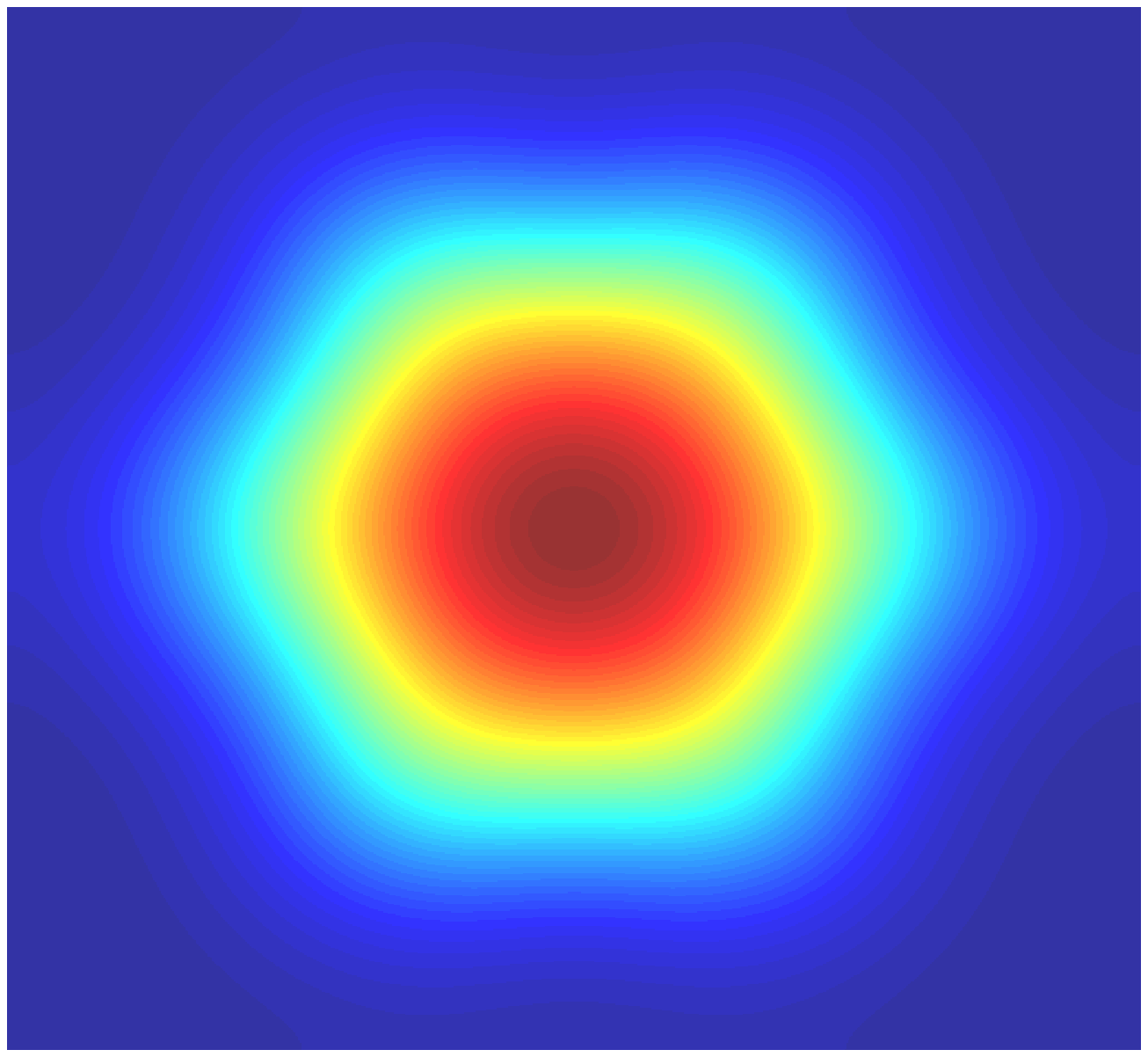} &
		\includegraphics[width=1.6in,height=1.2in]{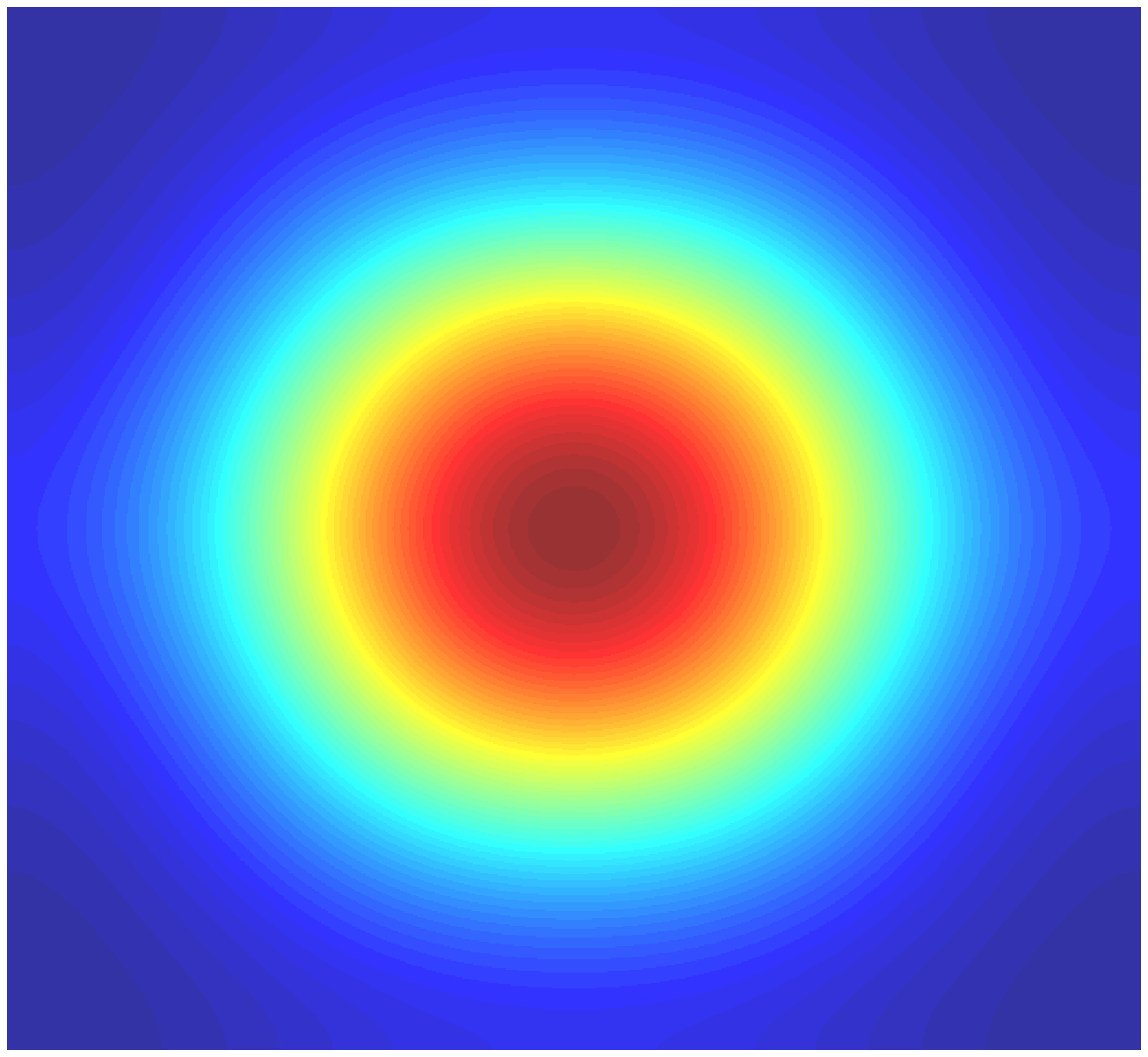} \\
		& $t = 0$ & $t = 0.005$ & $t = 0.01$ & $t = 0.02$ \\
		\rotatebox{90}{FRS} &
		\includegraphics[width=1.6in,height=1.2in]{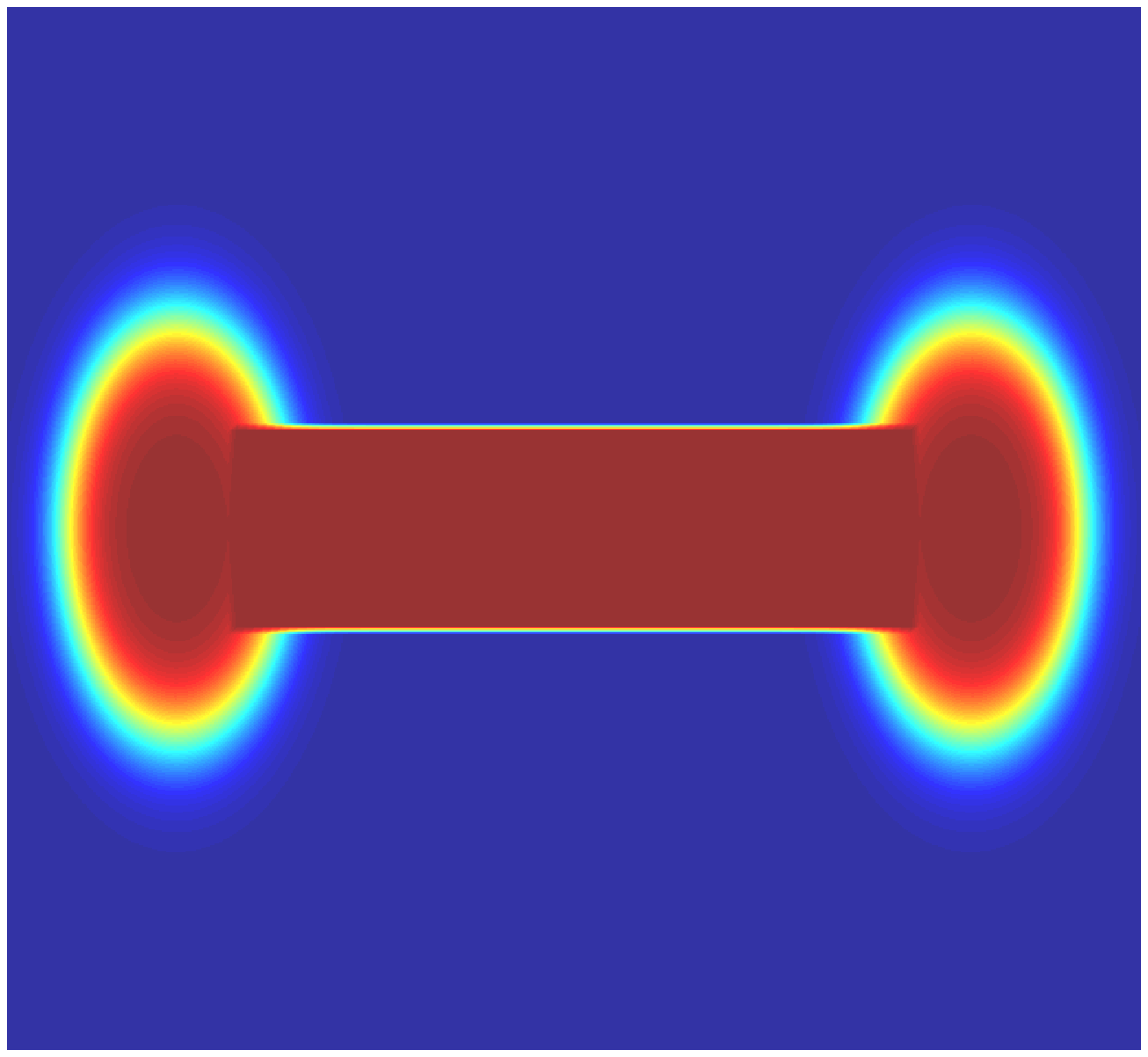} &
		\includegraphics[width=1.6in,height=1.2in]{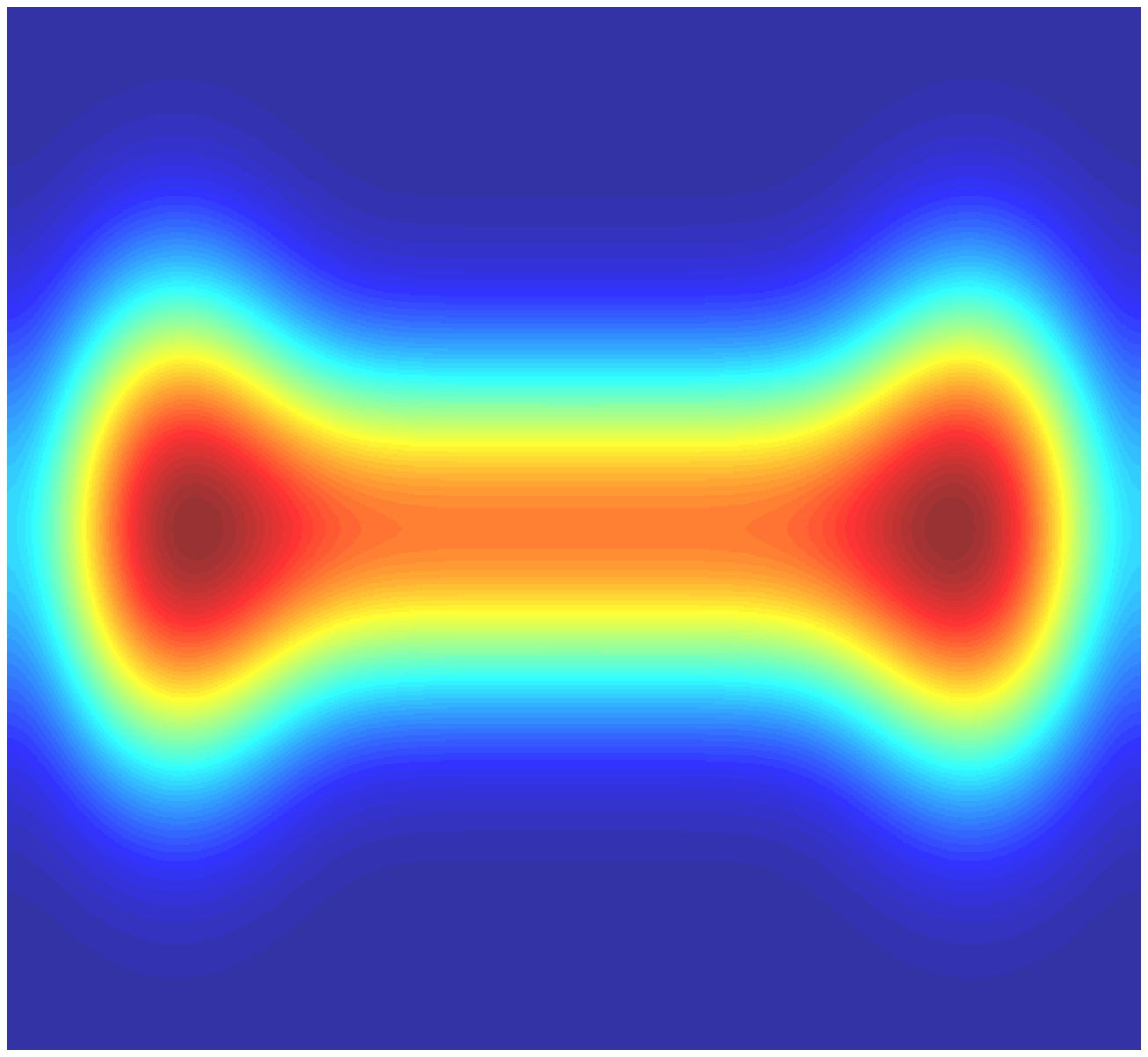} &
		\includegraphics[width=1.6in,height=1.2in]{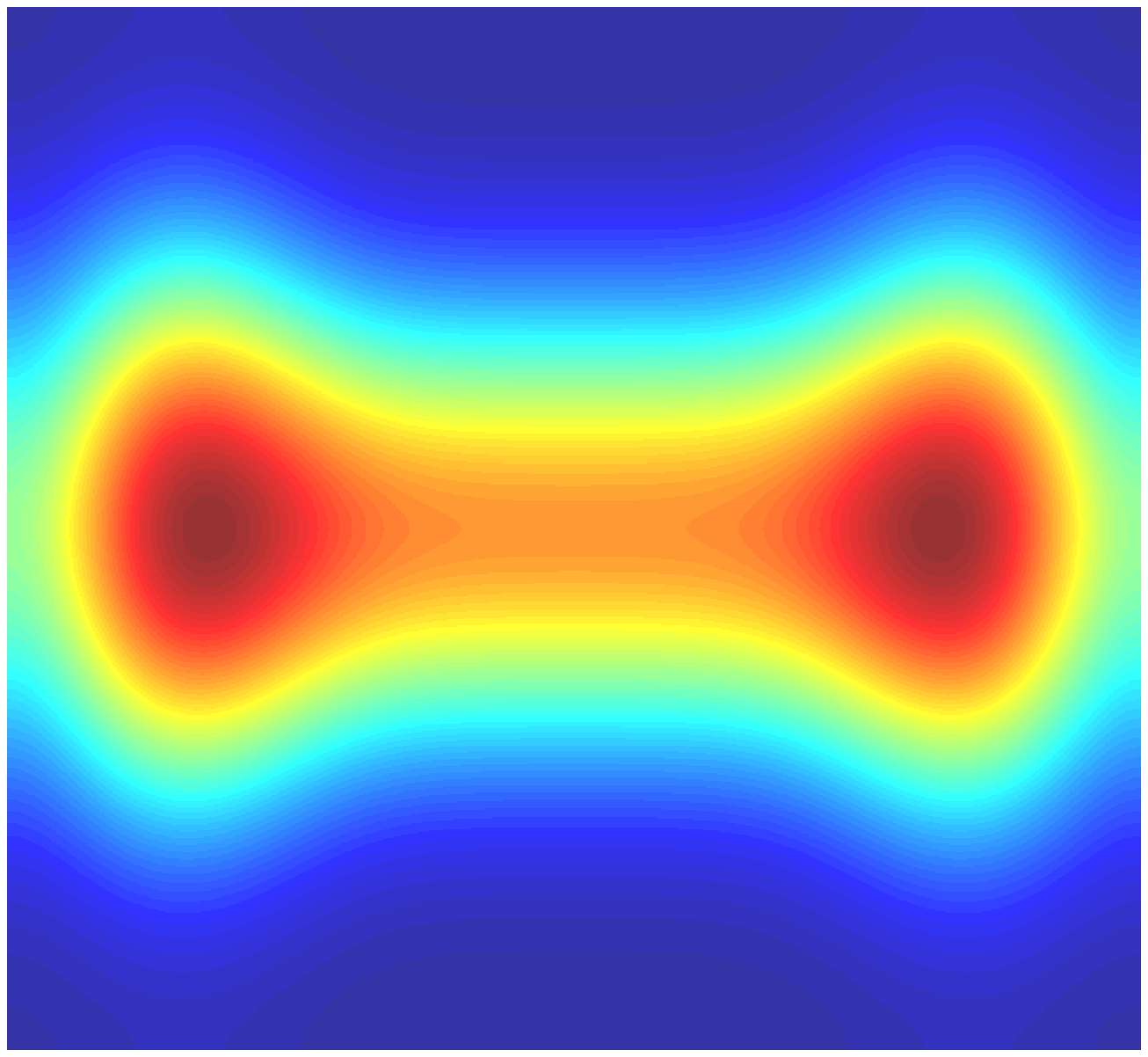} &
		\includegraphics[width=1.6in,height=1.2in]{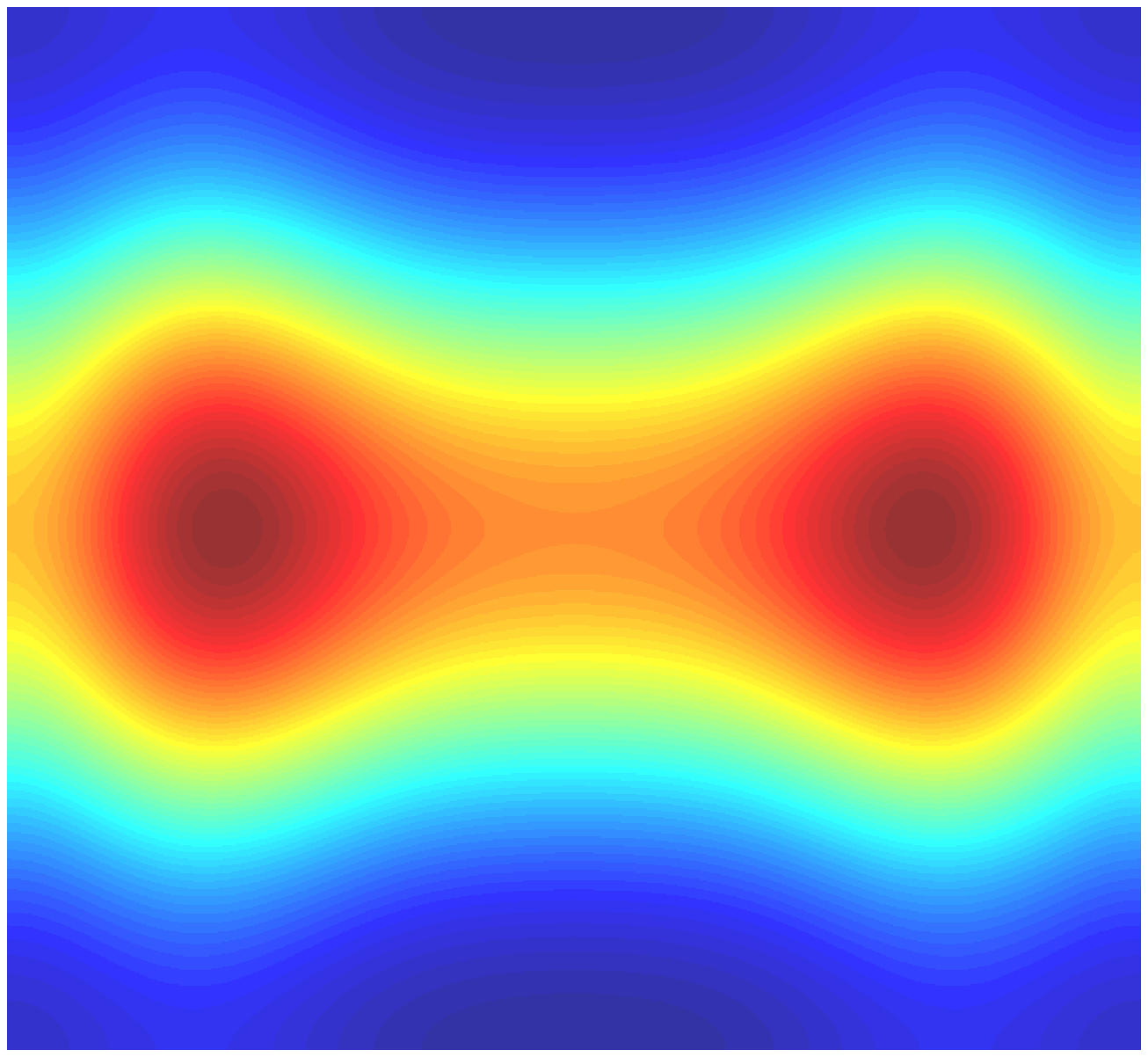} \\
		\rotatebox{90}{$r_0 = 6$} &
		\includegraphics[width=1.6in,height=1.2in]{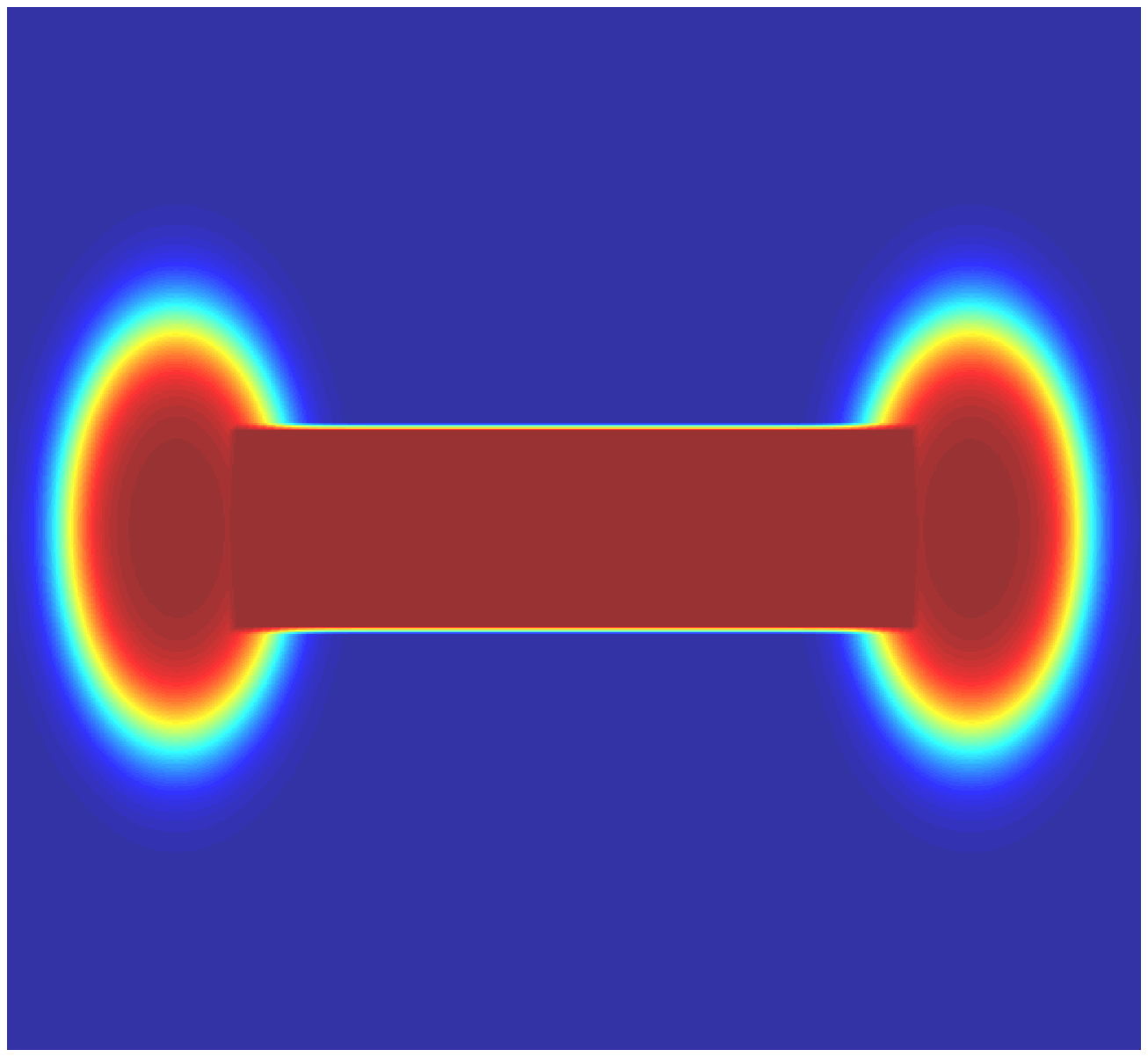} &
		\includegraphics[width=1.6in,height=1.2in]{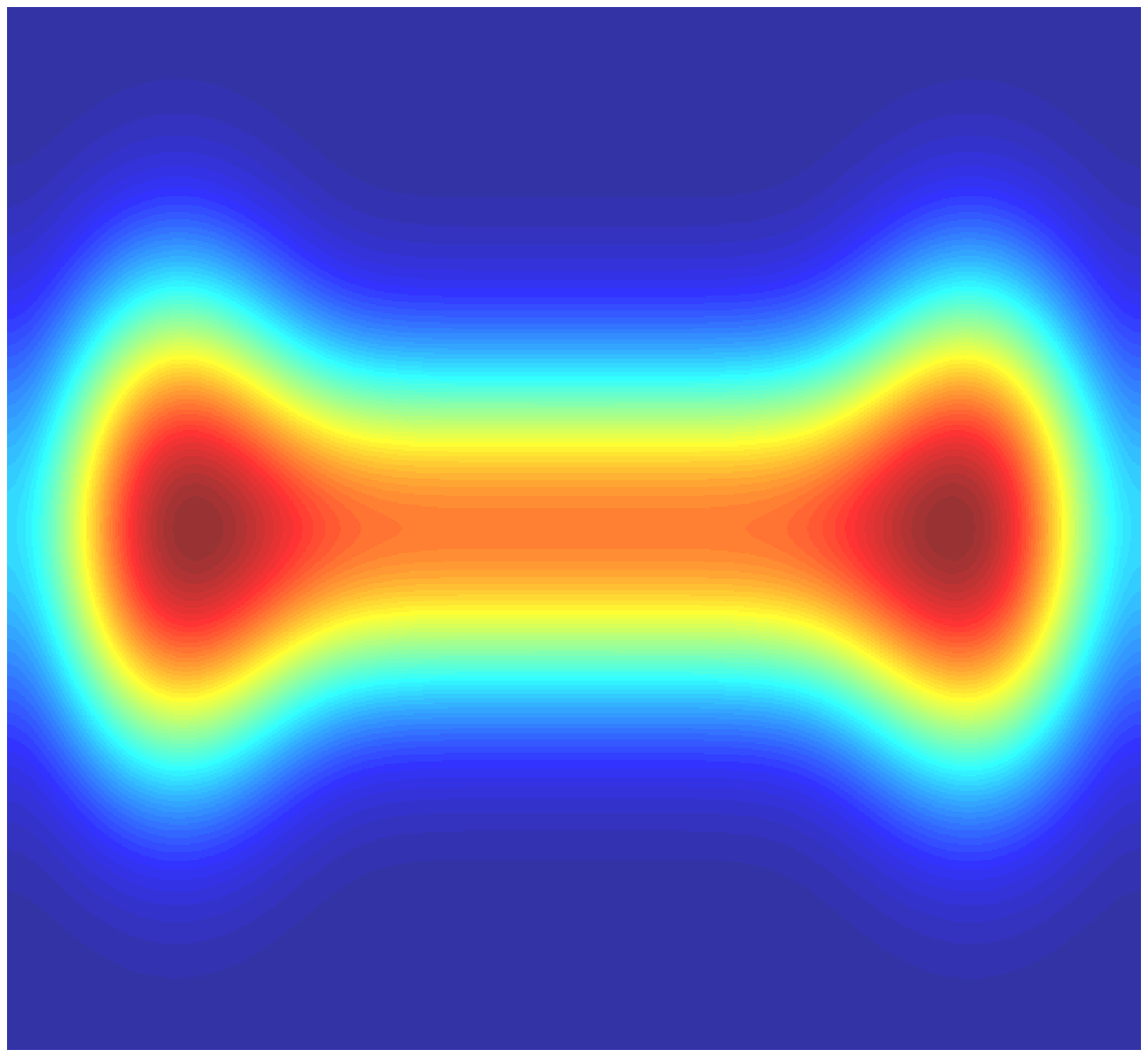} &
		\includegraphics[width=1.6in,height=1.2in]{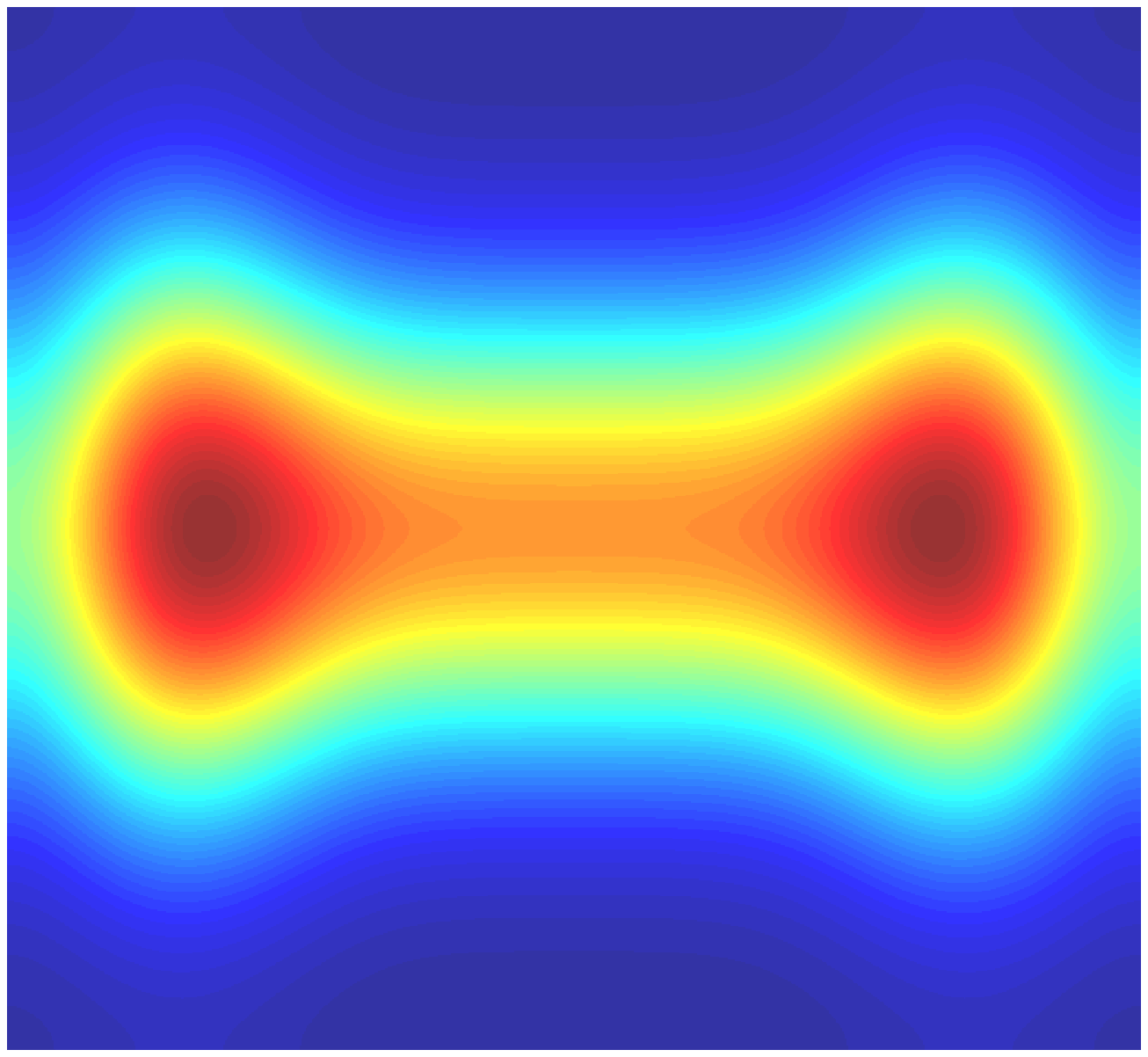} &
		\includegraphics[width=1.6in,height=1.2in]{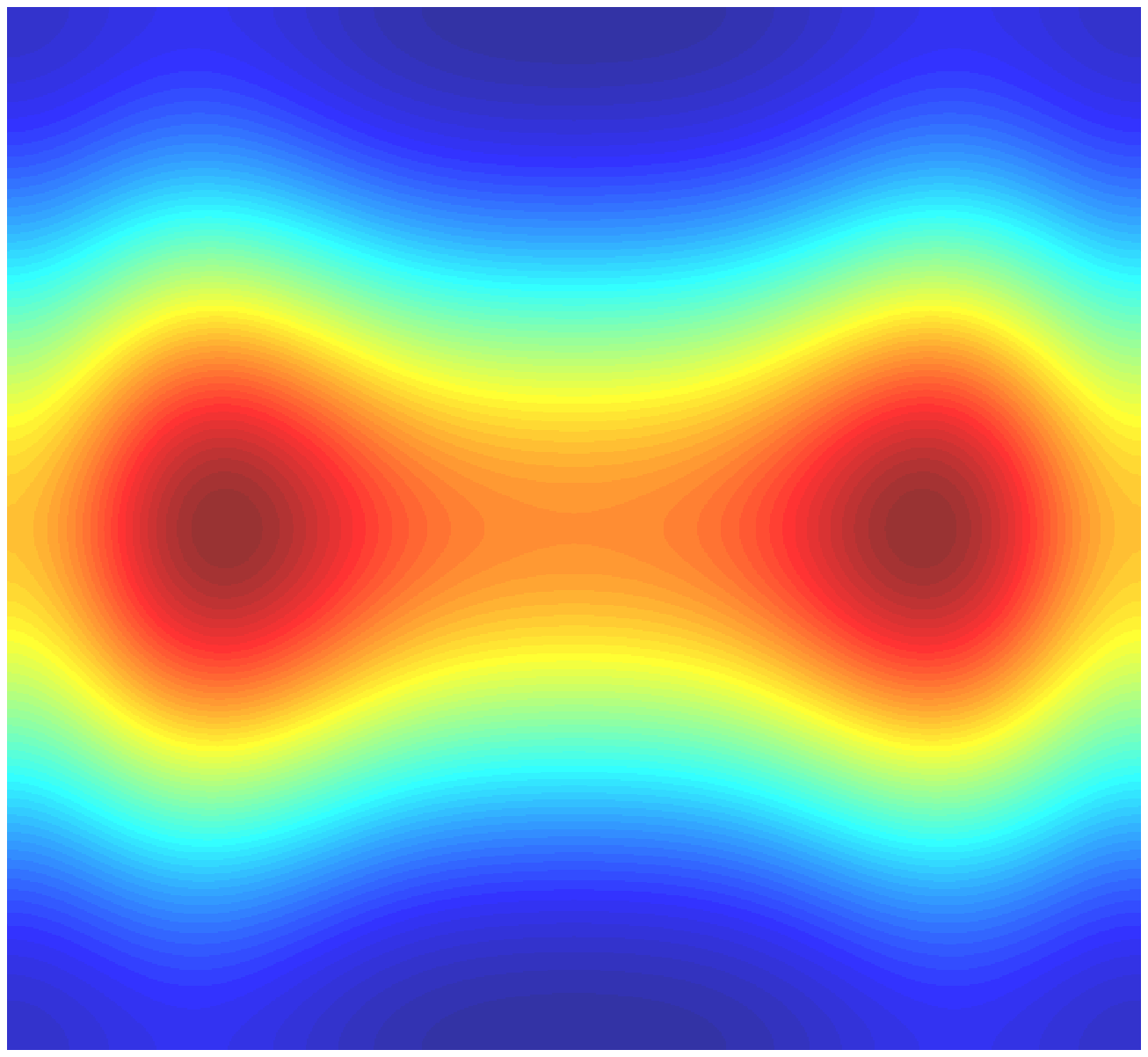} \\
		& $t = 0$ & $t = 0.005$ & $t = 0.01$ & $t = 0.04$ \\
		\rotatebox{90}{FRS} &
		\includegraphics[width=1.6in,height=1.2in]{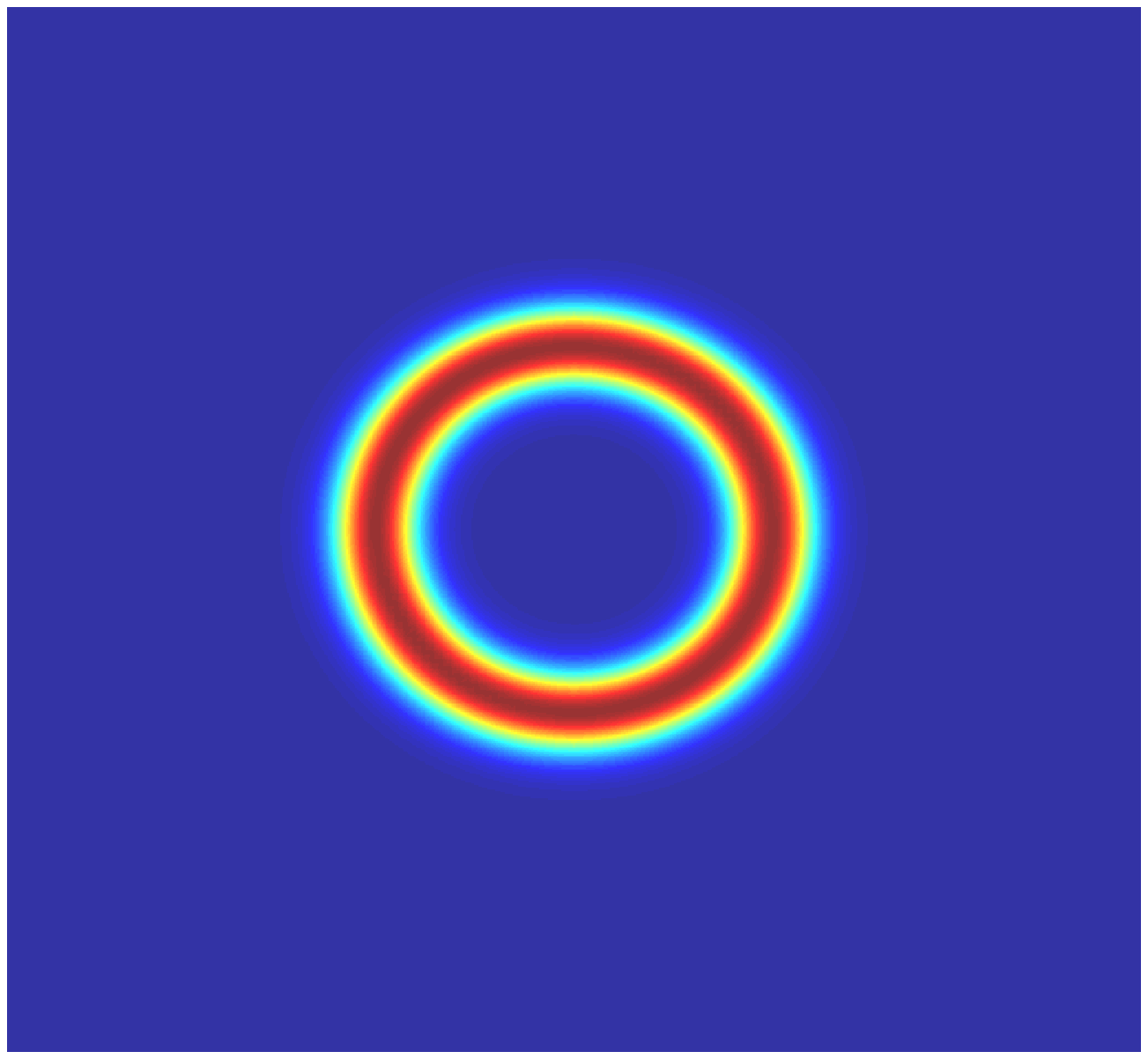} &
		\includegraphics[width=1.6in,height=1.2in]{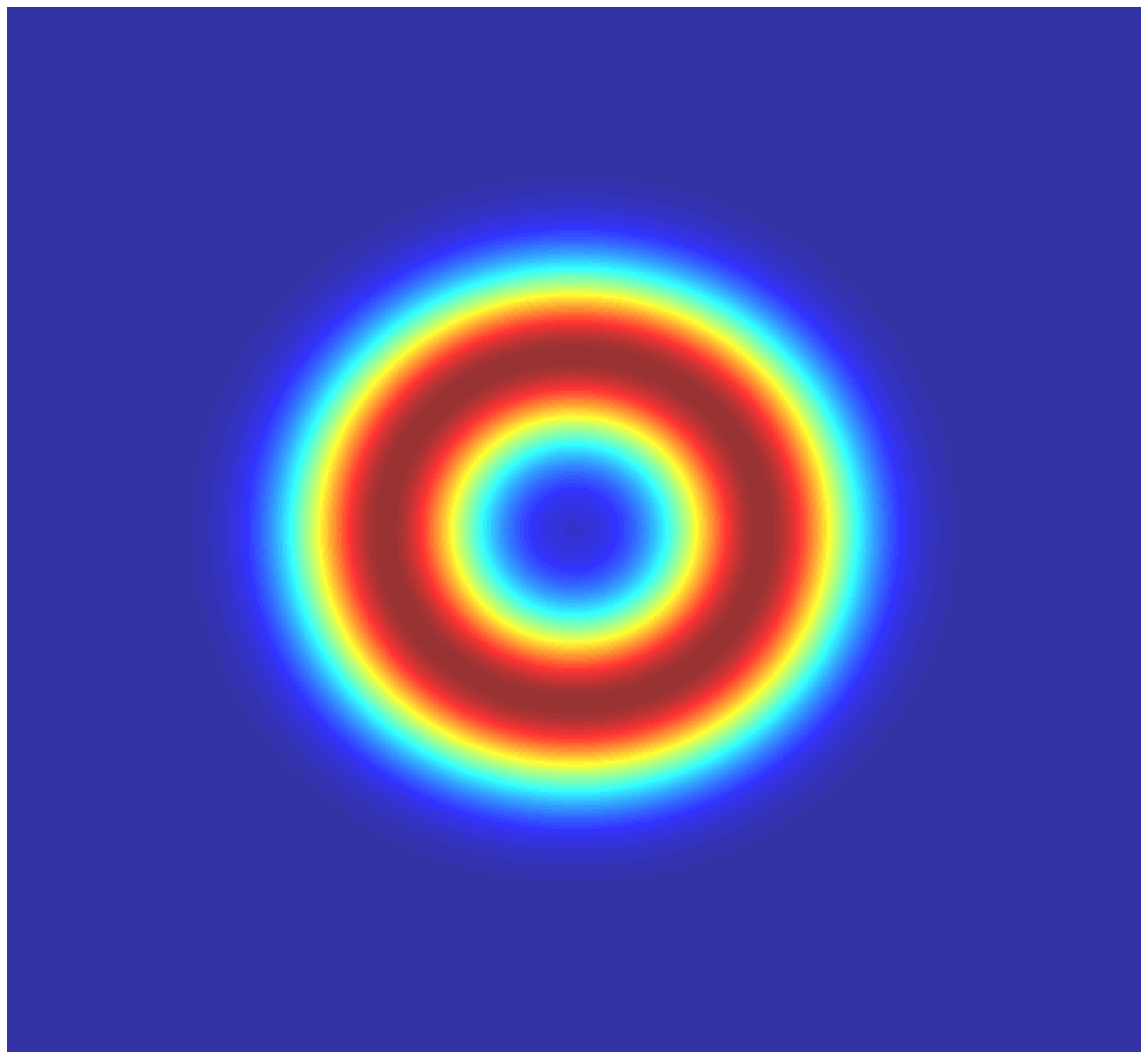} &
		\includegraphics[width=1.6in,height=1.2in]{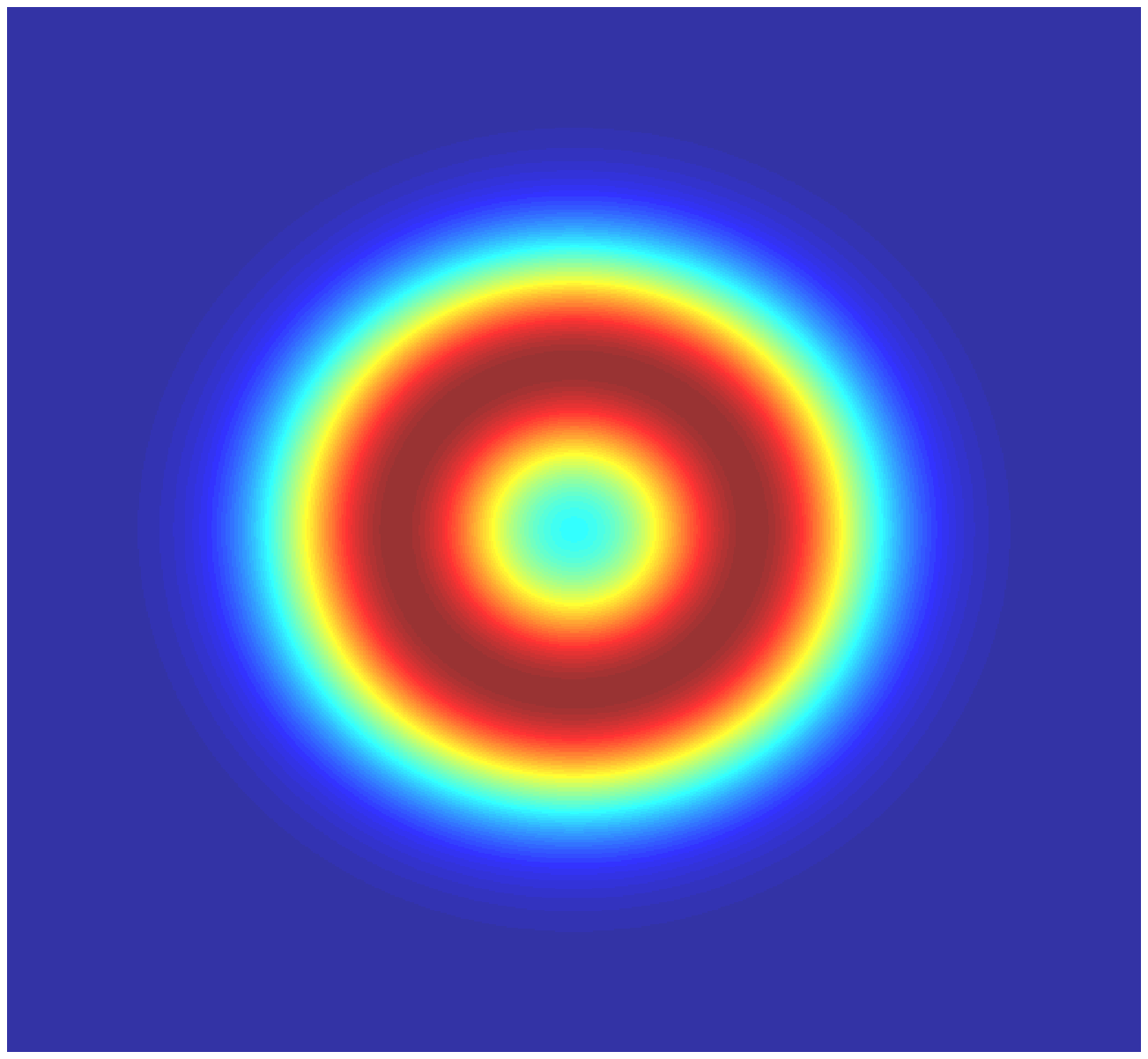} &
		\includegraphics[width=1.6in,height=1.2in]{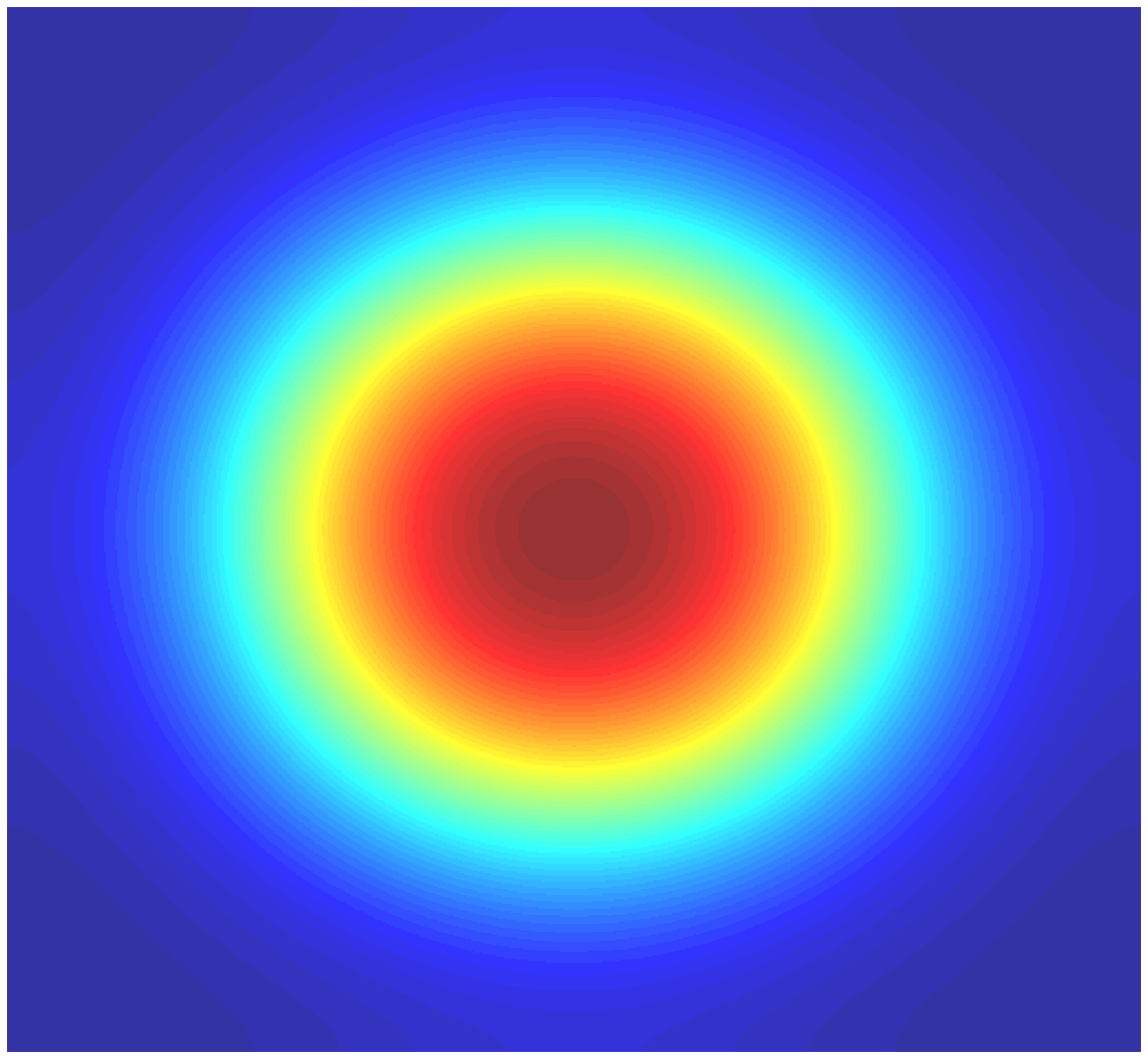} \\
		\rotatebox{90}{$r_0 = 10$} &
		\includegraphics[width=1.6in,height=1.2in]{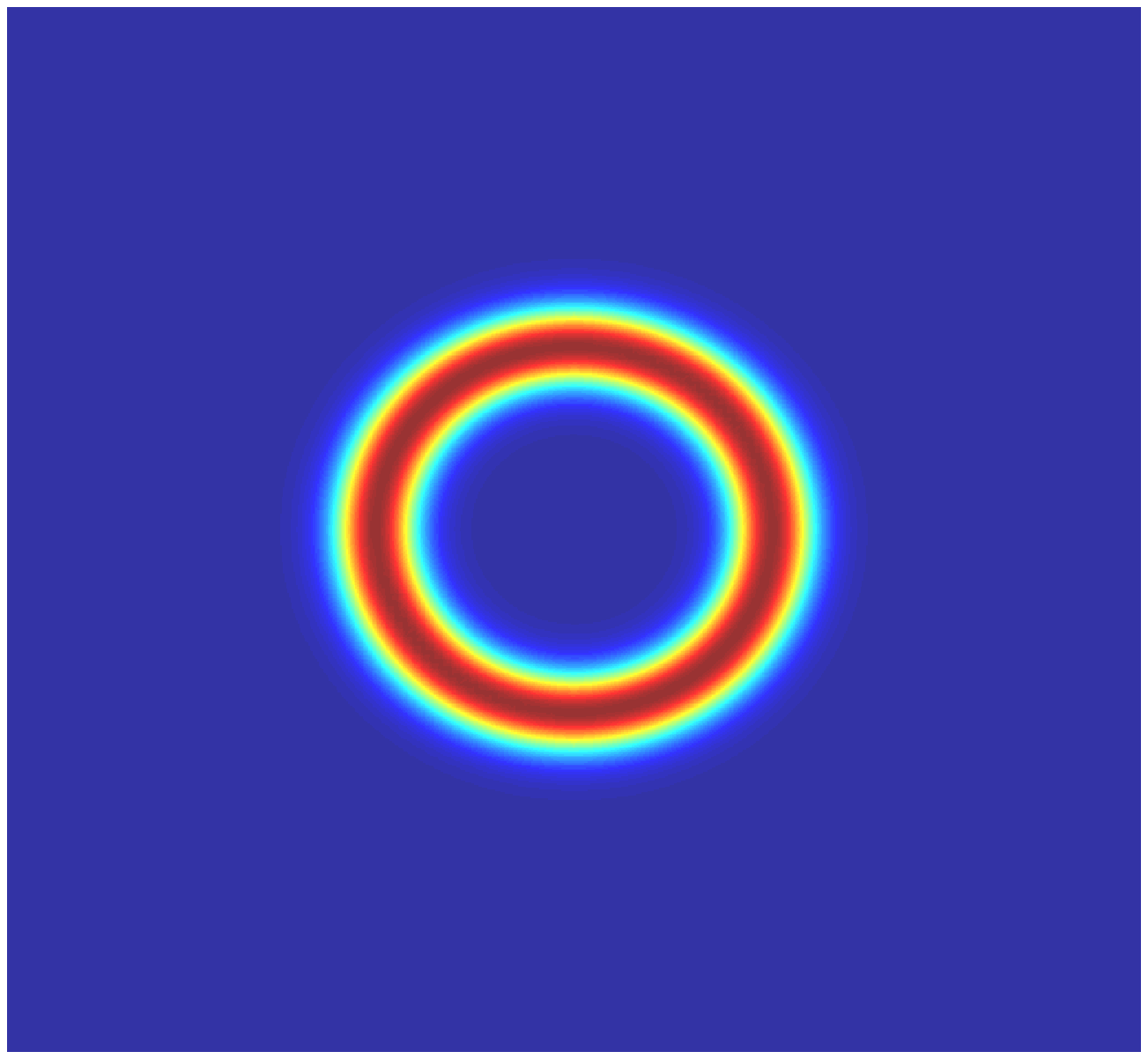} &
		\includegraphics[width=1.6in,height=1.2in]{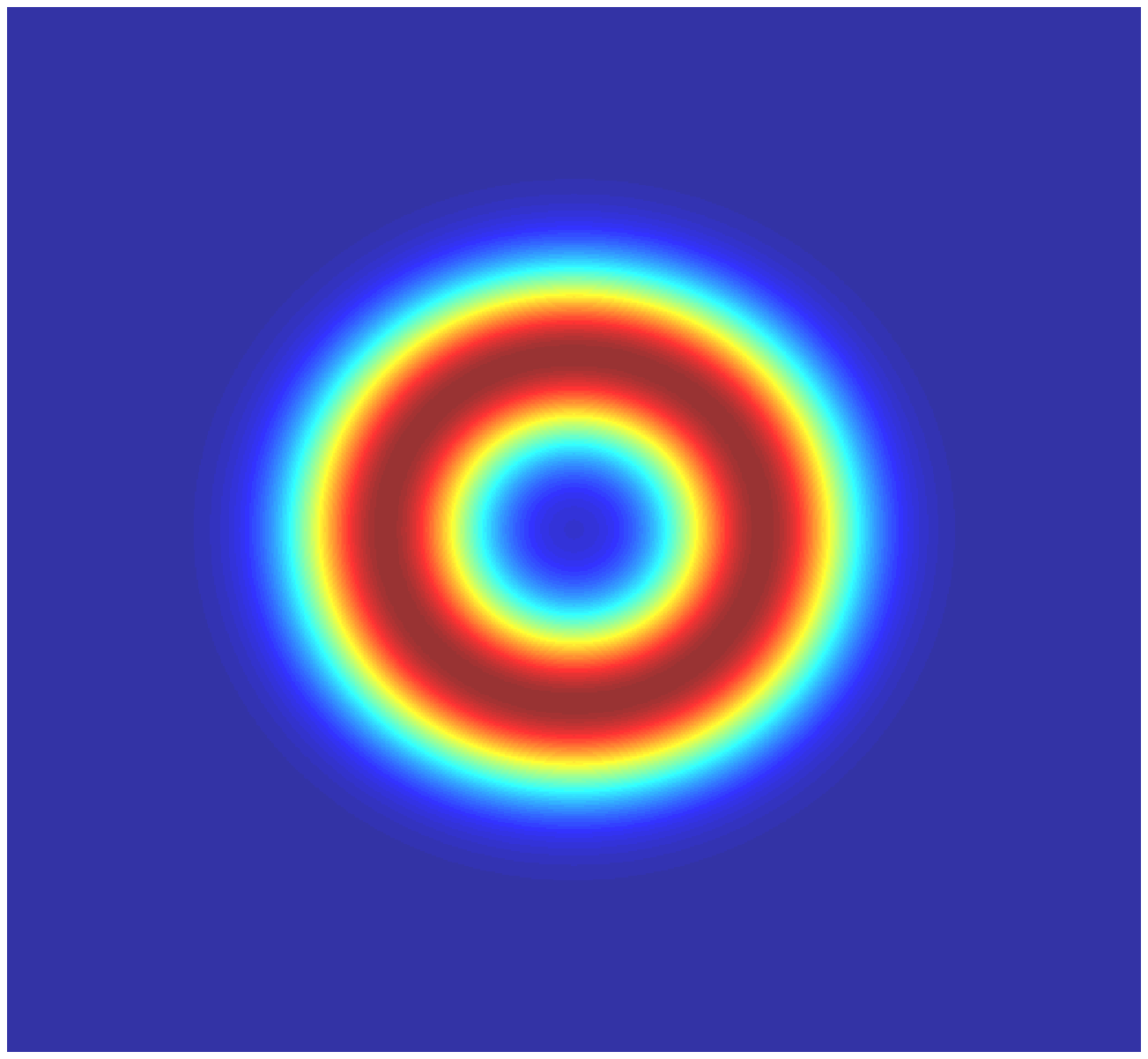} &
		\includegraphics[width=1.6in,height=1.2in]{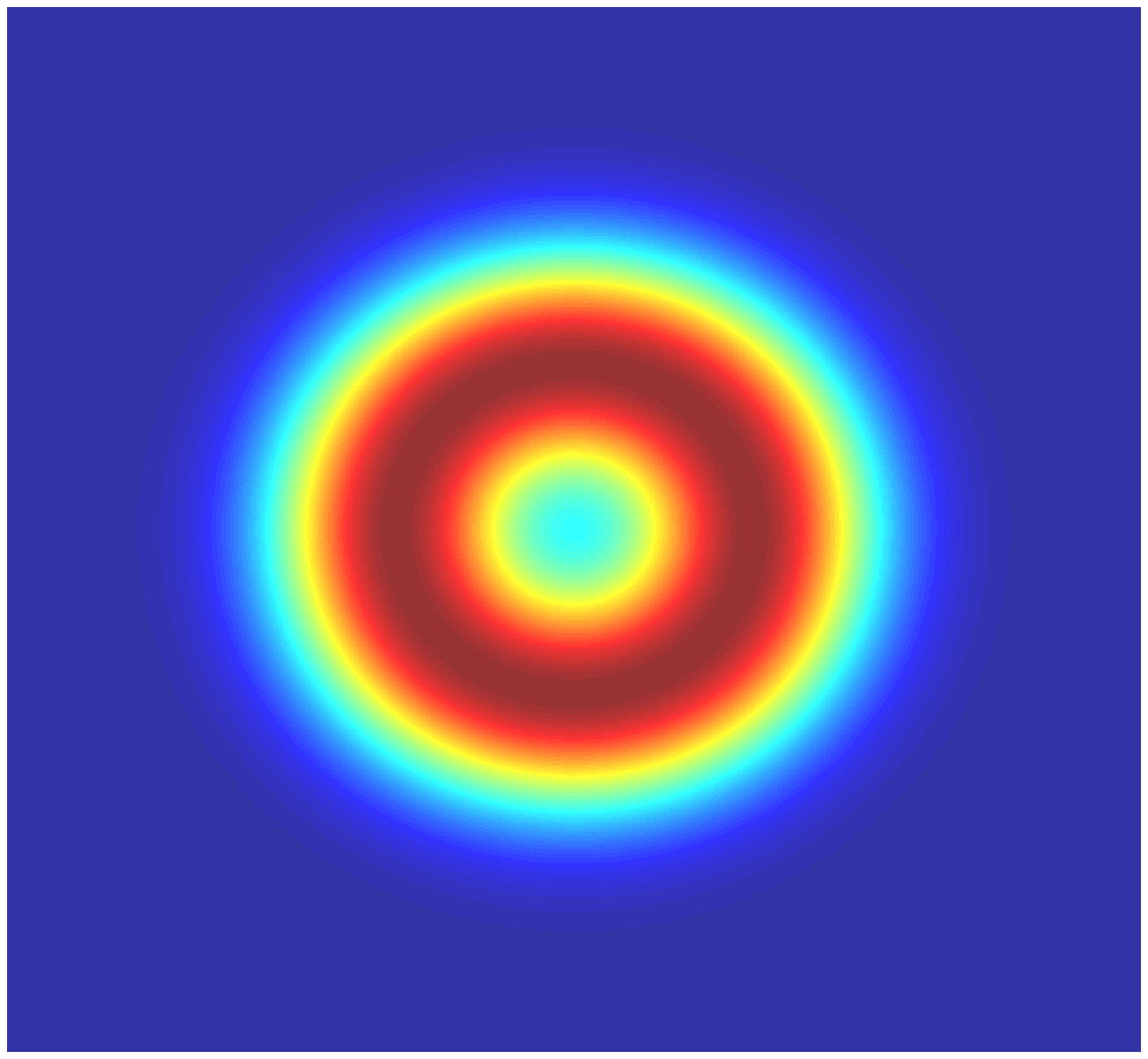} &
		\includegraphics[width=1.6in,height=1.2in]{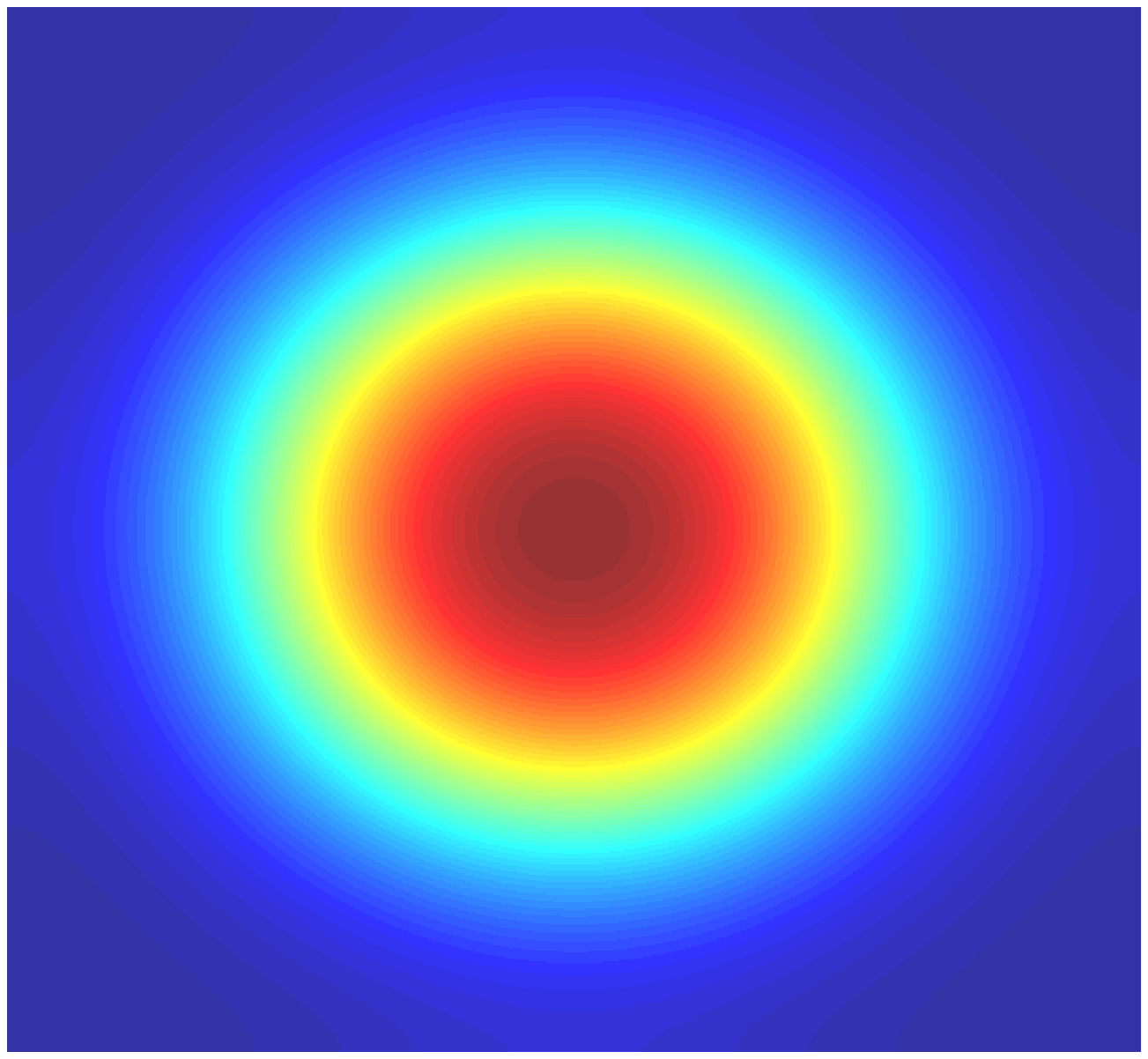} \\
	\end{tabular}
	\caption{Comparisons of the FRS and ALRS solutions for Example 2 with $M = N = 128$.
		First two rows: star. Middle two rows: dumbbell. Bottom two rows: torus.}
	\label{fig3}
\end{figure}

(2) Dumbbell:
\begin{equation*}
u(x,y,0) =
\begin{cases}
1, & 0.4 < x < 1.6, 0.4 < y < 0.6, \\
\begin{split}
1 + & \tanh \left( \frac{0.2 - \sqrt{(x - 0.3)^2 + (y - 0.5)^2}}{\varepsilon \sqrt{2}} \right) + \\
& \tanh \left( \frac{0.2 - \sqrt{(x - 1.7)^2 + (y - 0.5)^2}}{\varepsilon \sqrt{2}} \right), \end{split}
& \mathrm{others},
\end{cases}
\end{equation*}
where $(x,y) \in [0,2] \times [0,1]$.

(3) Torus: $u(x,y,0) = -1 +
	\tanh \left( \frac{0.4 - \sqrt{x^2 + y^2}}{\varepsilon \sqrt{2}} \right) -
	\tanh \left( \frac{0.3 - \sqrt{x^2 + y^2}}{\varepsilon \sqrt{2}} \right)$,
where $(x,y) \in [-1,1]^2$.

\noindent Moreover, in these three cases, we set $\varepsilon = \frac{5 h_x}{\sqrt{2} \tanh^{-1}(0.9)}$.

The temporal evolution of the three initial shapes are shown in Fig.~\ref{fig3}.
As can be seen that numerical solutions computed by
the FRS \eqref{eq2.5} and the ALRS \eqref{eq3.5} are not much difference.
We can see from this figure that for the star shape, the tips of it move inward
and the gaps move outward. Finally, the star changes to the shape of a circle.
For the dumbbell shape, the spheres at both ends are increasing
and the width of handle part becomes large.
For the torus shape, the inner circle shrinks faster than the outer circle.
Then, the inner circle disappears and resulting in a circle shape.
Fig.~\ref{fig4} displays the ranks of the numerical solutions computed by the two proposed methods
with three different initial shapes.
It can be seen in this figure that the rank of the low-rank numerical solutions
is lower than the full-rank numerical solutions.
The rank of the full-rank numerical solutions varies strongly at the first several time steps.
According to \cite{Ceruti2022adaptiveDLR}, rank adaptivity is more suitable for this problem.
\begin{figure}[h]
	\centering
	\subfigure[star, $t = 0.01$]
	{\includegraphics[width=2.1in,height=2.4in]{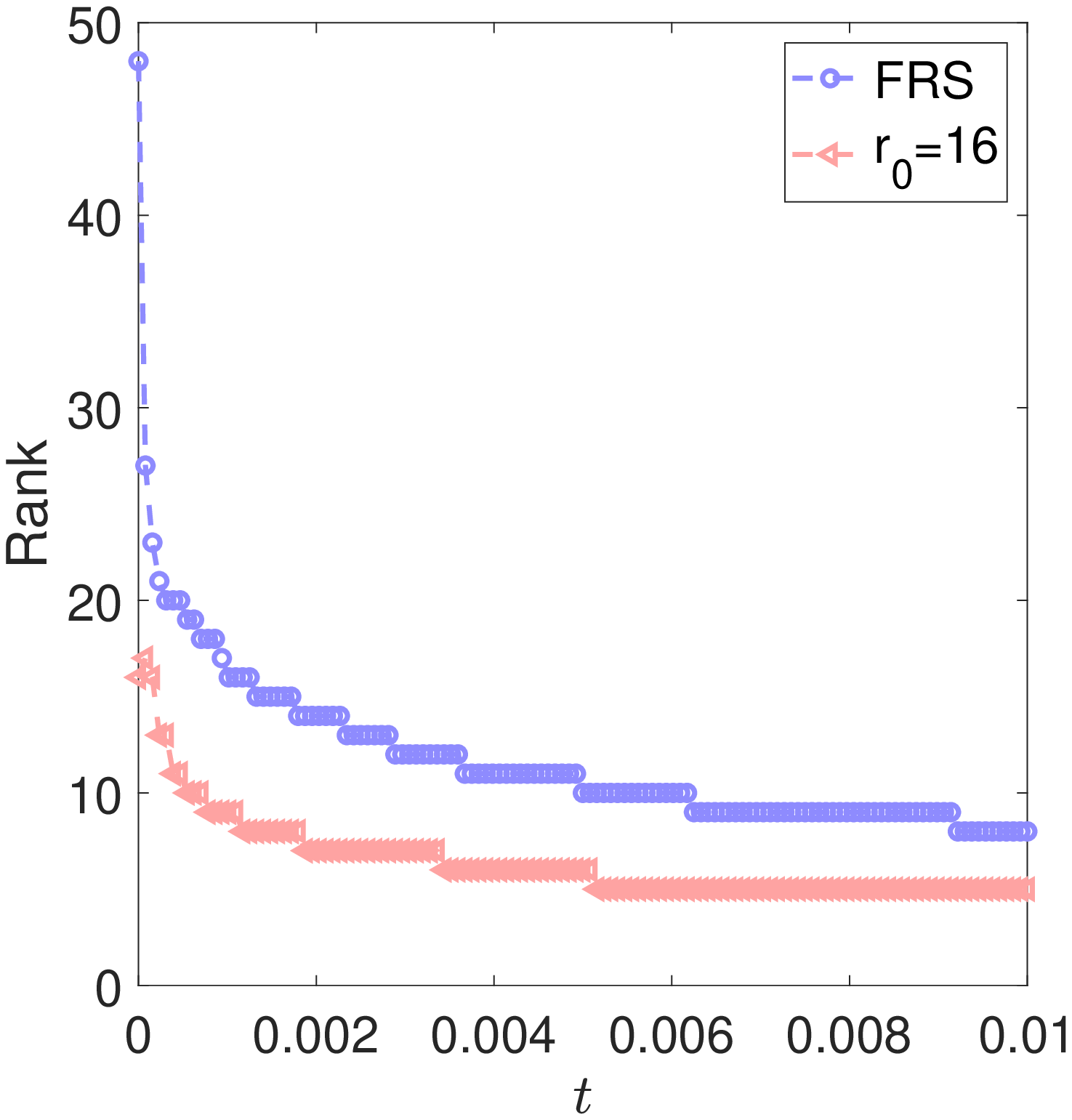}}
	\subfigure[dumbbell, $t = 0.02$]
	{\includegraphics[width=2.1in,height=2.4in]{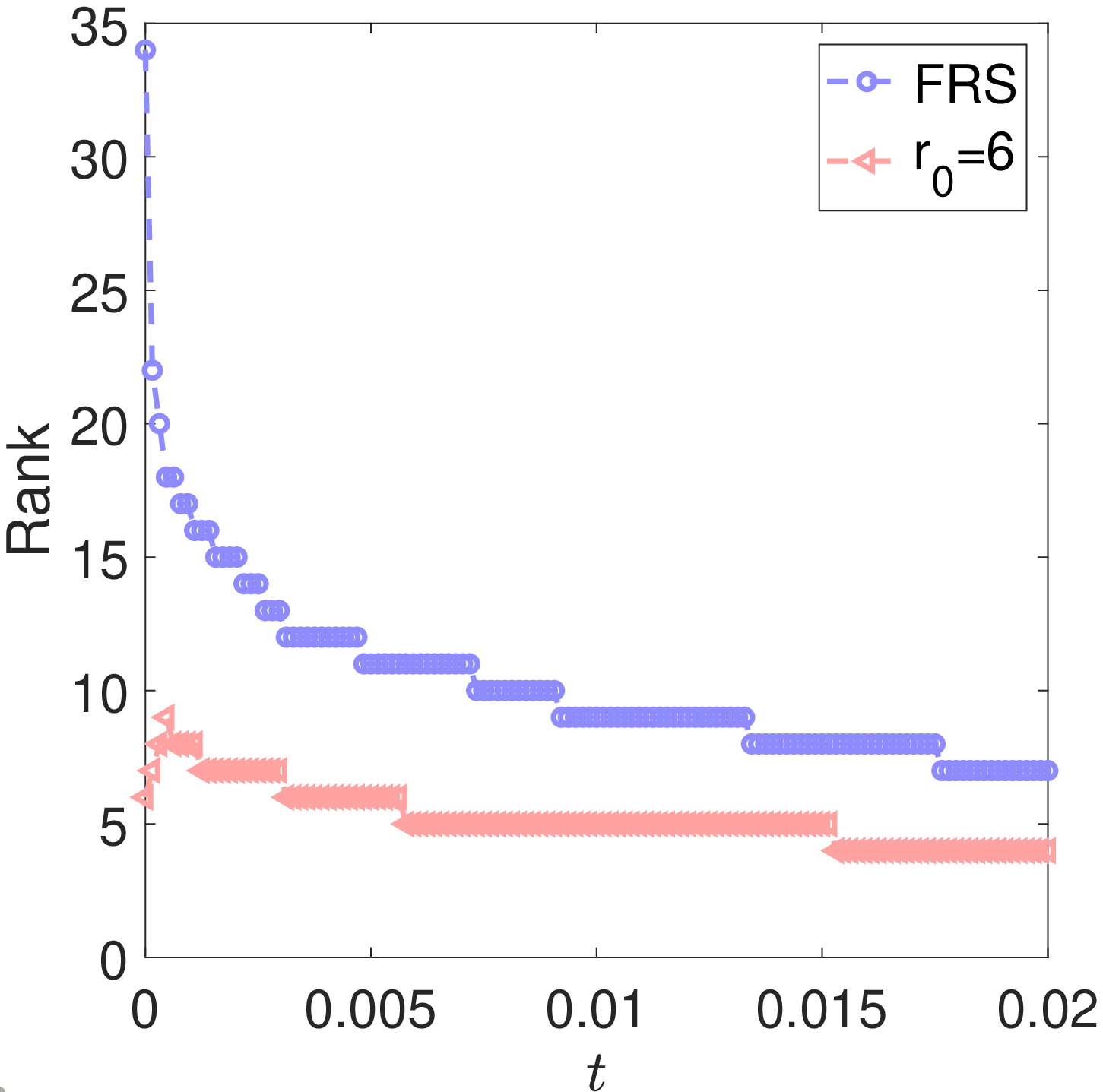}}
	\subfigure[torus, $t = 0.04$]
	{\includegraphics[width=2.1in,height=2.4in]{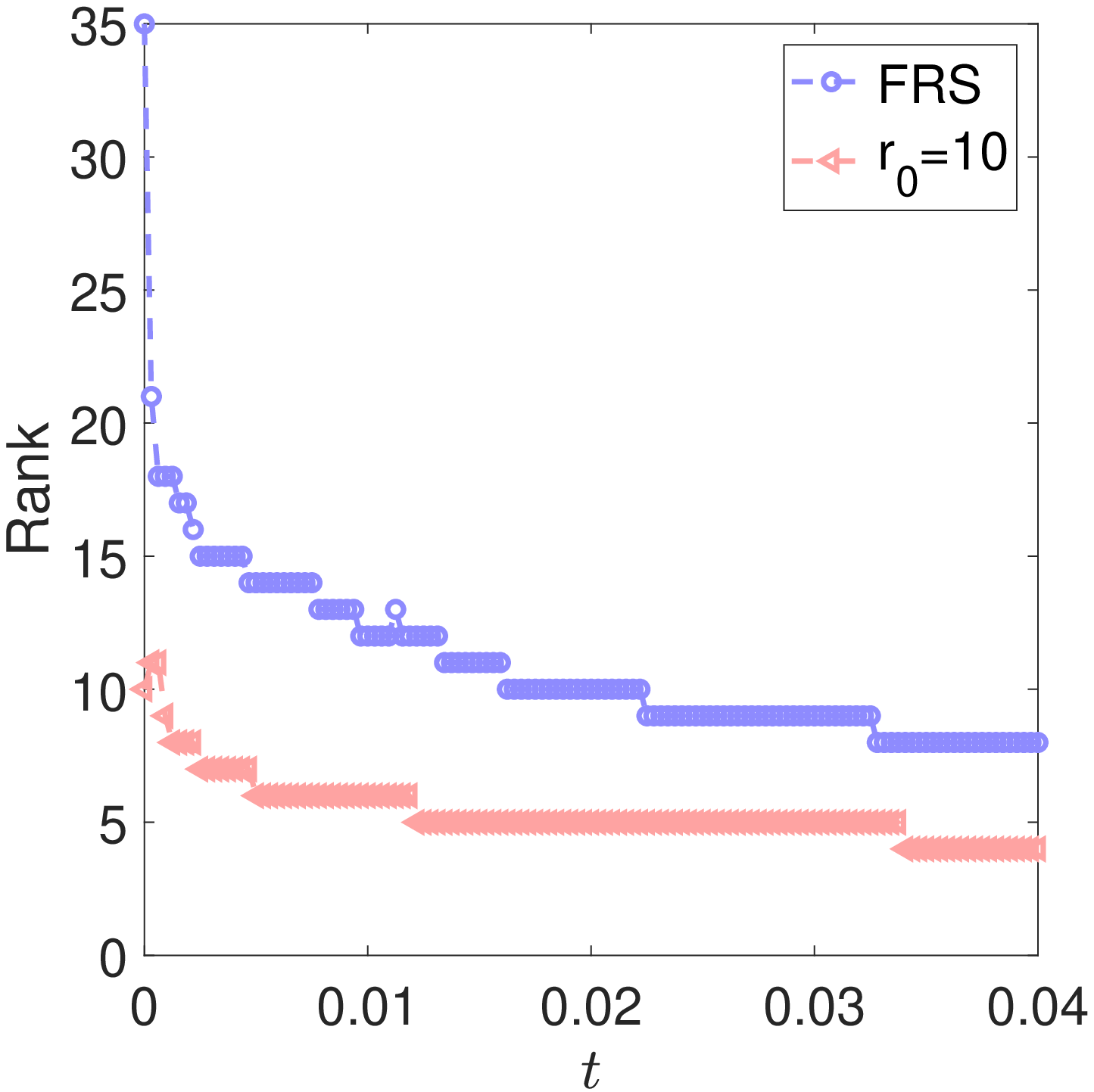}}
	\caption{The ranks of the numerical solutions
		computed by \eqref{eq2.5} and \eqref{eq3.5} for Example 2 with different initial shapes
		and $M = N = 128$.}
	\label{fig4}
\end{figure}

\section{Concluding remarks}
\label{sec5}

For solving the EFK equation \eqref{eq1.1} numerically,
we propose a FRS scheme \eqref{eq2.5} and a rank-adaptive splitting approach \eqref{eq3.5}.
Based on the semi-discrete system \eqref{eq2.1}, we first split it into three parts.
This yields our FRS scheme \eqref{eq2.5}.
Then, the convergence and the discrete maximum principle of the scheme are proved in Section \ref{sec2}.
According to the frame of the FRS scheme,
the ALRS approach \eqref{eq3.5} is designed for finding a low-rank solution of \eqref{eq1.1}.
Example 1 strongly supports the theoretical results giving in Section \ref{sec2}.
On the other hand, it also shows that our methods can preserve the energy decaying.
Moreover, the results in Example 1 indicate that for a given suitable initial rank $r_0$,
the convergence orders of the ALRS approach \eqref{eq3.5} are first-order accurate in time 
and second-order accurate in space, respectively.
In Example 2, we simulate mean curvature effect motion
for three initial shapes (i.e., star, dumbbell and torus).
Based on this work, two future research directions are suggested:
\begin{itemize}
	\item[(i)] {Study the convergence and the energy decaying of the ALRS approach \eqref{eq3.5};}
	\item[(ii)] {Design a low-rank algorithm for Eq.~\eqref{eq1.1} following the idea of  \cite{Einkemmer2019Vlasov}. Some high-order low-rank algorithms are worth being considered for Eq.~\eqref{eq1.1}.}
\end{itemize}

\section*{Acknowledgments}
\addcontentsline{toc}{section}{Acknowledgments}
\label{sec6}

\textit{This research is supported by the National Natural Science Foundation of China
(No.~11801463), the Natural Science Foundation of Sichuan Province (No.~2022NSFSC1815)
and the Applied Basic Research Program of Sichuan Province (No.~2020YJ0007).
X.-M. Gu is also supported by Guanghua Talent Project of Southwestern University of Finance and Economics.}

\section*{Data availability}
The datasets generated and analyzed during the current study are not publicly available
but are available from the authors on reasonable request.

\section*{Conflict of interests}
The authors declare that they have no known competing financial interests or
personal relationships that could have appeared to influence the work reported in this paper.

\section*{References}
\addcontentsline{toc}{section}{References}
\bibliography{references}

\end{document}